\documentclass[a4paper,11pt]{article}
\usepackage{amsfonts,amssymb,amsthm,cite,amsmath,amstext}
\usepackage{paralist}
\usepackage{color}
\usepackage[colorlinks,linkcolor=red,citecolor=blue]{hyperref}

\numberwithin{equation}{section}

\definecolor{plum}{rgb}{0.8,0.2,0.8}

\title{\textbf{Degeneracy of entire curves into higher dimensional complex manifolds}}
\author{\textsc{Ya Deng}}
\date{}
\begin{document}
\maketitle

\newtheorem{thm}{Theorem}[section]
\newtheorem{rem}{Remark}[section]
\newtheorem{problem}{Problem}[section]
\newtheorem{conjecture}{Conjecture}[section]
\newtheorem{lem}{Lemma}[section]
\newtheorem{cor}{Corollary}[section]
\newtheorem{dfn}{Definition}[section]
\newtheorem{proposition}{Proposition}[section]
\newtheorem{example}{Example}[section]

\def\oc{\mathcal{O}} \def\oce{\mathcal{O}_E} \def\xc{\mathcal{X}}
\def\ac{\mathcal{A}} \def\rc{\mathcal{R}} \def\mc{\mathcal{M}}
\def\wc{\mathcal{W}} \def\fc{\mathcal{F}} \def\cf{\mathcal{C_F^+}}
\def\jc{\mathcal{J}}
\def\ic{\mathcal{I}}
\def\kc{\mathcal{K}}
\def\cb{\mathbb{C}}
\def\vol{\operatorname{vol}}
\def\ord{\operatorname{ord}}
\def\Im{\operatorname{Im}}
\def\dm{\mathrm{d}}

\def\as{{a^\star}} \def\es{e^\star}

\def\tl{\widetilde} \def\tly{{Y}} \def\om{\omega}

\def\cb{\mathbb{C}} \def\nb{\mathbb{N}} \def\nbs{\mathbb{N}^\star}
\def\pb{\mathbb{P}} \def\pbe{\mathbb{P}(E)} \def\rb{\mathbb{R}}
\def\zbb{\mathbb{Z}}

\def\hb{\bold{H}} \def\fb{\bold{F}} \def\eb{\bold{E}}
\def\pbb{\bold{P}}

\def\nab{\overline{\nabla}} \def\n{|\!|} \def\spec{\textrm{Spec}\,}
\def\cinf{\mathcal{C}_\infty} \def\d{\partial}
\def\db{\overline{\partial}}
\def\hess{\sqrt{-1}\partial\!\overline{\partial}}
\def\zb{\overline{z}} \def\lra{\longrightarrow}

\begin{abstract}
Pursuing McQuillan's philosophy in proving the Green-Griffiths conjecture for certain surfaces of general type, we deal with the algebraic degeneracy of entire curves tangent to holomorphic foliations by curves. Inspired by the recent work \cite{PS14},  we study the intersection of Ahlfors current $T[f]$ with tangent bundle $T_\fc$ of $\fc$, and derive some consequences. In particular, we introduce the definition of weakly reduced singularities for foliations by curves, which requires less work than the exact classification for foliations.  Finally we discuss the strategy to prove the Green-Griffiths conjecture for complex surfaces.
\end{abstract}

\tableofcontents

\section{Introduction}
In \cite{McQ98}, 
McQuillan proved the following striking theorem:
\begin{thm}\label{McQ1}
Let $X$ be a surface of general type and $\fc$ a holomorphic foliation on $X$. 
Then any entire curve $\fc:\cb \rightarrow X$ tangent to $\fc$ can not be Zariski dense.
\end{thm}

The original proof of Theorem \ref{McQ1} is rather involved. Later on, several simplified proofs appeared, cf.\ \cite{Bru99} and \cite{PS14}. The idea in proving Theorem (\ref{McQ1}) is to argue by contradiction. One assumes that there exists a Zariski dense entire curve $f:\cb\rightarrow X$ which is tangent to $\fc$. Then, one studies the intersection of the Ahlfors current $T[f]$, which can be treated as a $(1,1)$-cohomology class in $X$, with the tangent bundle and the normal bundle of the foliation $\fc$. The above works proved that both of the intersections numbers are positive. However, since $K_X$ is big, then $T[f]\cdot K_X>0$, and by the equality $K^{-1}_X=T_\fc+N_\fc$, a contradiction is obtained.

In our paper, we give a generalization of McQuillan's result dealing with entire curves on higher dimensional complex manifolds, by pursuing the same philosophy. Firstly we obtain a formula for $T[f]\cdot T_\fc$, which is an improvement of that of~\cite{PS14}. Our first result is the following theorem:
\begin{thm}\label{formula tangent main}
	Let $(X,\fc)$ be a K\"ahler 1-foliated pair. 
	If $f:\cb\rightarrow X$ is a \emph{transcendental} entire curve (see Definition \ref{algebraic curve} below) tangent to $\fc$ whose image is not contained in $\text{Sing}(\fc)$, then
	\begin{equation}\nonumber
	\langle T[f], c_1(T_\fc)\rangle+T(f,\jc_{\fc})=\langle 
	T[f_1],\oc_{X_1}(-1)\rangle\geq0,
	\end{equation}
	where $\jc_{\fc}$ is a coherent ideal sheaf determined by the singularity of $\fc$, and $T(f,\jc_{\fc})$ is a non-negative real number representing the ``intersection" of $T[f]$ with $\jc_{\fc}\,;$ this number will be defined later.
\end{thm}
If $X$ is a complex surface, as is proved by McQuillan, Theorem \ref{formula tangent main} can be improved to the extent that
\begin{equation}\label{positive tangent}
T[f]\cdot T_\fc\geq0.
\end{equation}
On the one hand, by pursuing  his philosophy of ``diaphantine approximation", one can generalize (\ref{positive tangent}) to higher dimensional manifolds, under some assumptions on the foliation:
\begin{thm}
	Let $\fc$ be a foliation by curves on the $n$-dimensional complex manifold $X$, 
	such that the singular set $\text{Sing}(\fc)$ of the foliation $\fc$ is a set of absolutely isolated singularities (this will be defined later). If $f:\cb\rightarrow X$ is an Zariski dense entire curve which is tangent to $\fc$, then one can blow-up $X$ a finite number of times to get a new birational model $(\widetilde{X},\widetilde{\fc})$ such that
	$$
	T[\widetilde{f}]\cdot T_{\widetilde{\fc}}=0.
	$$
	
	Moreover, if $\fc$ is a foliation by curves with absolutely isolated singularities that are simple (see Definition \ref{simple singularity}), and whose canonical bundle $K_\fc$ is big, then every entire curve $f:\cb\rightarrow X$ that is tangent to $\fc$ is algebraically degenerate.
\end{thm}

 On the other hand,  Theorem \ref{formula tangent main} leads us to the fact that the error term $T(f,\jc_{\fc})$ is controllable, if the singularities of $\fc$ are not too ``bad'' (they are called \emph{weakly reduced singularities}  in Section \ref{normal section}). The theorem is as follows:

\begin{thm}\label{negative normal main}
	Let $X$ be a projective manifold of dimension $n$ endowed with a 1-dimensional foliation 
	$\fc$ with weakly reduced singularities. If $f$ is a Zariski dense entire curve tangent to $\fc$, satisfying $\langle T[f], 
	K_X\rangle>0$ (e.g. $K_X$ is big), then we have
	$$
	\langle T[\hat{f}], c_1(N_{\hat{\fc}})\rangle<0
	$$
	for some birational pair $(\hat{X},\hat{\fc})$.
\end{thm}

\begin{rem}{\rm
	Our definition of ``weakly reduced singularities" is actually weaker than the usual concept of reduced singularities, which always requires a lot of checking (e.g.\ a classification of singularities). We only need to focus on the multiplier ideal sheaf of $\jc_\fc$, instead of trying to understand the exact behavior of singularities.}
\end{rem}

It is notable that the following strong result due to M. Brunella implies a conclusive contradiction in combination with Theorem \ref{negative normal main}, in the case of dimension 2.
\begin{thm}\label{brunella}
	Let $X$ be a complex surface endowed with a foliation $\fc$ (no assumption is made for singularities of $\fc$ here). If $f:\cb\rightarrow X$ is a Zariski dense entire curve tangent to $\fc$, then we have
	$$
	\langle T[f],N_{\fc}\rangle\geq0.
	$$ 
\end{thm}
Therefore, we get another proof of Theorem \ref{McQ1} without using the refined formula (\ref{positive tangent}) immediately. This leads us to observe that if one can resolve any singularities of the 1-dimensional foliation $\fc$ into weakly reduced ones, and generalize the previous Brunella Theorem to higher dimensional manifolds, one could infer the Green-Griffiths conjecture for surfaces of general type. 

\begin{thm}\label{Green}
Assume that Theorem \ref{brunella} holds for a directed variety $(X,\fc)$ where $X$ is a base of arbitrary dimension and $\fc$ has rank $1$, and that one can resolve the singularities of $\fc$ into weakly reduced ones. Then every entire curve drawn in a projective surface of general type must be algebraically degenerate.
\end{thm}

\section{Proof of the Main Theorems}

\subsection{Notions and definitions}\label{notions}
For any coherent ideal sheaf $\jc\subset \oc_X$, one can construct a global 
quasi-plurisubarmonic function $\varphi_\jc$ on $X$ such that 
$$
\varphi_\jc=\log(\sum_i|g_i|^2)+\oc(1)
$$
where $(g_i)$ are local holomorphic functions that generate the ideal $\jc$. We 
call $\varphi_\jc$ the \emph{characteristic function associated to the coherent 
sheaf}. For any entire curve $f$ which is not contained in the subscheme $Z(\jc)$, we 
can write
$$
\varphi_\jc\circ f(t)|_{B(r)}=\sum_{|t_j|< r} \nu_j\log|t-t_j|^2+\oc(1),
$$
and here we call $\nu_j$ the \emph{multiplicity of $f$ along $\jc$}. 

In a related way, we define the \emph{proximity function} of $f$ with respect to $\jc$ by
$$
m_{f,\jc}(r):=-\frac{1}{2\pi}\int_{0}^{2\pi}\varphi_\jc\circ f(re^{i\theta}) d\theta,
$$
and the \emph{counting function} of $f$ with respect to $\jc$ by
$$
N_{f,\jc}(r):=\sum_{|t_j|<r}\nu_j\log\frac{r}{|t_j|}.
$$

Let us take a log resolution of $\fc$ with a birational morphism 
$p:\hat{X}\rightarrow X$ such that $p^{-1}(\jc)=\oc_{\hat{X}}(-D)$, let
$\hat{f}$ denote the lift of $f$ to $\hat{X}$ (so that $p\circ \hat{f}=f$), 
and let $\Theta_D$ be the curvature form of $D$.

Now we recall the following formula, which will be very useful in what follows.
\begin{thm}(Jensen formula)
	For $r\geq 1$ we have
	\begin{equation}\label{Jensen}
		\int_{1}^{r}\frac{dt}{t}\int_{B(t)}dd^c\varphi=\frac{1}{2\pi}\int_{0}^{2\pi}\varphi(re^{
			i\theta})d\theta-\frac{1}{2\pi}\int_{0}^{2\pi}\varphi(e^{i\theta})d\theta,
	\end{equation}
	in particular if $\varphi$ is a quasi-plurisubharmonic function, then for $r$ 
	large enough we have
	$$
	\int_{1}^{r}\frac{dt}{t}\int_{B(t)}dd^c\varphi=\frac{1}{2\pi}\int_{0}^{2\pi}\varphi(re^{
		i\theta})d\theta+\oc(1).
	$$
\end{thm}

Then the following \emph{First Main Theorem} due to Nevanlinna is an immediate consequence.

\begin{thm}\label{first}As $r\rightarrow \infty$, one has
$$T_{\hat{f},\Theta_D}(r)=N_{f,\jc}(r)+m_{f,\jc}(r)+\oc(1).$$
\end{thm}

Let $\fc$ be a 1-dimension foliation. Then we can take an open covering 
$\{U_\alpha\}_{\alpha\in I}$ such that on each $U_\alpha$ there exists $v_\alpha\in 
H^0(U_\alpha,T_X|_{U_\alpha})$ which generates $\fc$, and such that the 
$v_\alpha$ coincide up to multiplication by nowhere vanishing holomorphic 
functions $\{g_{\alpha\beta}\}$:
$$
v_\alpha=g_{\alpha\beta}v_\beta
$$
if $U_\alpha\cap U_\beta \neq \emptyset$. The functions $\{g_{\alpha\beta}\}$ 
define a cohomology class $H^1(X,\oc^*_X)$ that corresponds to the cotangent 
bundle of $\fc$, denoted here by $T^*_\fc$. Let $\omega$ be a hermitian 
metric on $X$. Then $\omega$ induces a natural singular metric 
$h_s$ on $T_\fc$. Indeed, on 
each $U_\alpha$ the local weight $\varphi_\alpha$ is given by
\begin{equation}\label{singular metric}
\varphi_\alpha=-\log |v_\alpha|^2_\omega=-\log 
\sum_{i,j}a^i_\alpha\overline{a^j_\alpha}\omega_{i\overline{j}},
\end{equation}
where $v_\alpha=\sum_{i=1}^{n}a^i_{\alpha}\frac{\partial}{\partial z^i}$ with 
respect to the coordinate system $(z_1,\ldots,z_n)$ on $U_\alpha$.

We are going to define a coherent sheaf $\jc_{\fc}$ reflecting the behavior of the singularies of $\fc$: on each $U_\alpha$ the generators of $\jc_{\fc}$ 
are precisely the coefficients $(a^i_\alpha)$ of the vector $v_\alpha$ defining 
$\fc$, i.e.
$$
v_\alpha=\sum_{i=1}^{n}a^i_\alpha\frac{\partial}{\partial z^i}. 
$$
If we fix a smooth metric $h$ on $T_\fc$, then there exists a globally defined 
function $\varphi_s$ such that 
$$
h=h_se^{-\varphi_s}
$$  
We know that
\begin{equation}\label{metric}
\varphi_s=\log|v_\alpha|^2_\omega
\end{equation}
modulo a bounded function, and $\varphi_s$ is the characteristic function associated to 
the coherent sheaf~$\jc_{\fc}$.

All the constructions explained above can be generalized to log pairs. We first begin with the following definition.
\begin{dfn}
Let $X$ be a smooth K\"ahler manifold, $D$ a simple normal crossing divisor 
and $\fc$ a foliation by curves defined on $X$. We say that $\fc$ is defined on the log pair $(X,D)$ if each component of $D$ is invariant by $\fc$. For brevity we say that $(X,\fc,D)$ a \emph{K\"ahler 1-foliated triple}.
\end{dfn}

The logarithmic tangent bundle $T_X\langle -\log D\rangle$ with respect to the 
pair $(X,D)$ is the subsheaf of $T_X$ whose local sections are given by
$$
v=\sum_{j=1}^{k}z_jv_j\frac{\d}{\d z_j}+\sum_{i=k+1}^{n}v_i\frac{\d}{\d z_i},
$$
where $z_1z_2\ldots z_k=0$ is the local equation of $D$.

We can take a hermitian metric $\omega_{X,D}$ on $T_X\langle -\log D\rangle$ 
defined as follows：
$$
\omega_{X,D}=\sqrt{-1}\sum_{i,j=1}^{k}\omega_{i\bar{j}}\frac{dz_i\wedge 
d\bar{z}_j}{z_i\bar{z}_j}+2 \text{Re}\sqrt{-1}\sum_{i>k\geq 
j}\omega_{i\bar{j}}\frac{dz_i\wedge 
d\bar{z}_j}{\bar{z}_j}+\sqrt{-1}\sum_{i,j\geq k+1}\omega_{i\bar{j}}dz_i\wedge 
d\bar{z}_j,
$$
where $(\omega_{i\bar{j}})$ is smooth positive definite hermitian matrix. In 
other words, the local model of $\omega_{X,D}$ is given by
$$
\omega_{X,D}\equiv \sqrt{-1}\sum_{i=1}^{k}\frac{dz_i\wedge 
d\bar{z}_i}{|z_i|^2}++\sqrt{-1}\sum_{i\geq k+1}dz_i\wedge d\bar{z}_i.
$$
Assume that the foliation $\fc$ is defined on $(X,D)$. Locally we have
$$
v_\alpha=\sum_{j=1}^{k}z_ja_{\alpha}^j\frac{\d}{\d 
z_j}+\sum_{i=k+1}^{n}a_{\alpha}^i\frac{\d}{\d z_i}
$$
as the generator of $\fc$. Then $\omega_{X,D}$ induces a singular hermitian 
metric $h_{s,D}$ on $T_\fc$ whose local weight is given by
$$
\varphi_{\alpha,D}=-\log |v_\alpha|^2_{\omega_{X,D}}=-\log 
\sum_{i,j}a^i_\alpha\overline{a^j_\alpha}\omega_{i\overline{j}}.
$$
We denote by $\jc_{\fc,D}$ the coherent sheaf defined by the functions 
$(a_{\alpha}^j)$. In general we have
$$
\jc_{\fc}\subset \jc_{\fc,D},
$$
and the inclusion may be strict. If we find a smooth metric 
$h=h_{s,D}e^{-\varphi_{s,D}}$ on $T_\fc$, then it is easy to check that 
$\varphi_{s,D}$ is the characteristic function associated with $\jc_{\fc,D}$. 

We set $\bar{X}_1:=P(T_X\langle -\log D\rangle)$, then $\omega_{X,D}$ induces a 
natural smooth metric $h_1$ on $\oc_{\bar{X}_1}(-1)$.

\subsection{Basic results about Ahlfors currents}

Let $X$ be a compact complex manifold, and let $f$ be an entire curve. Then we can associate a closed positive current of $(n-1,n-1)$ type as follows:
for any smooth 2-form $\eta$, we let
$$
\langle T_{r_k}[f],\eta\rangle:=\frac{T_{f,\eta}(r_k)}{T_{f,\omega}(r_k)},
$$
where $T_{f,\omega}(r):=\int_{1}^{r}\frac{dt}{t}\int_{B(t)}f^*\omega$. Here $B(t)$ is the ball of radius $t$ in $\cb$ and $S(t)$ is its boundary. It is 
known that one can find a suitable sequence of $(r_k)$ that tends to infinity,
such that the weak limit of $T_{r_k}[f]$ is a closed positive current (ref. \cite{Bru99}). It is 
denoted by $T[f]$ and called the \emph{Ahlfors current of $f$}.
Indeed, to ensure that the weak limit of $T_{r_k}[f]$ is closed, we only need $r_k$ to satisfy the following condition
\begin{equation}\label{ahlfors current}
\lim\limits_{r_k\rightarrow \infty}\frac{\text{length}(f(S(r_k)))}{\text{area}(f(B(r_k)))}=0.\tag{*}
\end{equation}
\begin{rem}\label{semi}{\rm
In the definition of the Ahlfors current, it is not indispensable to assume 
$\omega$ to be a K\"ahler form. In fact, it suffices to assume that $\omega$ is 
a semi-positive form satisfying
$$\lim\limits_{r_k\rightarrow 
\infty}\frac{T_{f,\omega}(r_k)}{T_{f,\widetilde{\omega}}(r_k)}>C>0$$
with respect to some K\"ahler form $\widetilde{\omega}$.}
\end{rem}

\begin{thm}\label{nef}
Let $L$ be a big line bundle on a K\"ahler manifold $X$. If $f:\cb\rightarrow X$ is an entire curve on $X$ such that its image is not contained in the augmented base locus $\mathbf{B}_+(L)$ of $L$ (ref.\ \cite{Laz04}), then $\langle T[f],c_1(L)\rangle>0$.
\end{thm}
\begin{proof}
	Since the image of $f$ is not contained in $\mathbf{B}_+(L)$, by the definition of the augmented base locus one can find an effective divisor $E$ whose support does not contain the image of $f$, such that
	$$
	L\equiv A+E,
	$$
	where $A$ is an $\mathbb{Q}$ ample divisor, and ``$\equiv$'' means \emph{numerically equivalent}.
	Then the counting function of $f$ with respect to $E$ is non-negative and one can find a smooth hermitian metric $h_E$ on $E$ such that the proximity function of $f$ with respect to $E$ is also non-negative. Therefore 
	$$
	\langle T[f], \Theta_{h_E}(E)\rangle \geq 0.
	$$
	By the ampleness of $A$, we have
	$$
	\langle T[f], c_1(A)\rangle > 0,
	$$
	and thus 
	$$
	\langle T[f], c_1(L)\rangle=\langle T[f], c_1(A)+c_1(E)\rangle > 0.
	$$
\end{proof}

\begin{dfn}\label{algebraic curve}
	An entire curve $f:\cb\rightarrow X$ is said to be an \emph{algebraic curve} iff $f$ admits a factorization in the form $f=g\circ R$, where $R:\cb\rightarrow \mathbb{P}^1$ is a rational function and $g:\mathbb{P}^1\rightarrow X$ is a rational curve; otherwise $f$ will be called \emph{transcendental}.
\end{dfn}

We have the following criterion for an entire curve to be algebraic (ref.\ \cite{Dem97}):
\begin{thm}\label{dense}
Any entire curve $f:\cb\rightarrow X$ is an algebraic curve if and only if $T_{f,\omega}(r)=\oc(\log r)$. In 
particular, if $f$ is Zariski dense, then
$$
\lim\limits_{r\rightarrow\infty}\frac{T_{f,\omega}(r)}{\log r}=+\infty.
$$
\end{thm}
From now on, for brevity we write $T_f(r)$ in place of $T_{f,\omega}(r)$, where 
$\omega$ can be some semi-positive $(1,1)$-form satisfying the condition in 
Remark \ref{semi}. 

The following logarithmic derivative lemma will be very useful in our arguments.
\begin{thm}(Logarithmic derivative lemma)
Let $f$ be a meromorphic function on $\mathbb{C}$. Then
\begin{eqnarray}\label{log}
\frac{1}{2\pi}\int_{0}^{2\pi} \log^+ |\frac{f'(re^{i\theta})}{f(re^{i\theta})}|d\theta \leq \oc(\log T_f(r)+\log r).
\end{eqnarray}
\end{thm}

Let $(X,V)$ be a smooth directed variety. It is easy to see that there is a canonical lift of $f$ to $X_1:=P(V)$, defined 
by $f_1(t):=(f(t),[f'(t)])$ and satisfying $\pi(f_1)=f$, where $\pi:X_1\rightarrow X$ is the natural projection map. Fix a hermitian metric 
$\omega$ on $X$. It induces a smooth metric $h$ on the line bundle $\oc_{X_1}(1)$. 
Then $0<\delta\ll 1$ , $\omega_1:=\pi^*\omega+\delta 
\Theta_h(\oc_{X_1}(1))$ is a hermitian metric on $X_1$, and we have the following 
lemma:
\begin{lem}\label{high tower}
Assume that $f$ is a transcendental entire curve on $X$. Let $\pi:X_1\rightarrow X$ be the projection map. Then we have
$$
\lim\limits_{r\rightarrow 
\infty}\frac{T_{f_1,\pi^*\omega}(r)}{T_{f_1,\omega_1}(r)}\geq 1.
$$
In particular, we can define the Ahlfors current $T[f_1]$ with respect to 
$\pi^*\omega$ in such a way that $\pi_*T[f_1]=T[f]$.
\end{lem} 
\begin{proof}
	Since $f'(\tau)$ can be seen as a section of the bundle $f_1^*(\oc_{X_1}(-1))$, 
	by the Lelong-Poincar\'e formula we have
\begin{equation}\label{poin}
dd^c\log 
|f'(\tau)|^2_\omega=\sum_{|t_j|<r}\mu_j\delta_{t_j}-f_1^*\Theta_{h^*}(\oc_{X_1}(-1))
\end{equation}
	on $B(r)$, where $\mu_j$ is the vanishing order of $f'(t)$ at $t_j$. Thus we get
	\begin{eqnarray}\label{tower}\nonumber
	\int_{1}^{r}\frac{dt}{t}\int_{B(t)}f_1^*\Theta_{h^*}(\oc_{X_1}(-1))=\sum_{|t_j|<r}
	\mu_j\log\frac{r}{|t_j|}-\int_{1}^{r}\frac{dt}{t}\int_{B(t)}dd^c\log 
	|f'(\tau)|^2_\omega\\ 
	=\sum_{|t_j|<r}\mu_j\log\frac{r}{|t_j|}-\frac{1}{2\pi}\int_{0}^{2\pi}\log|f'(re^
	{i\theta})|^2_\omega 
	d\theta+\frac{1}{2\pi}\int_{0}^{2\pi}\log|f'(e^{i\theta})|^2_\omega d\theta,
	\end{eqnarray}
	where the last equality is a consequence of the Jensen formula (\ref{Jensen}). Let $(\varphi_\alpha)_{\alpha\in J}$ be a partition of unity subcoordinate to the covering $(U_\alpha)_{\alpha\in J}$ of $X$. We can take a finite family of logarithms of 
	global meromorphic fuctions $(\log u_{\alpha j})_{\alpha\in J, 1\leq j\leq n}$ as local coordinates for $U_\alpha$, and by using the logarithmic 
	derivative lemma (\ref{log}) we have
\begin{eqnarray}\nonumber
\frac{1}{2\pi}\int_{0}^{2\pi}\log^+|f'(re^{i\theta})|^2_\omega 
d\theta&=&\sum_{\alpha\in J}\frac{1}{2\pi}\int_{0}^{2\pi}\varphi_\alpha\log^+|f'(re^{i\theta})|^2_\omega 
d\theta\\\nonumber
&\leq& \sum_{\alpha\in J}\sum_{j=1}^{n} C\int_{0}^{2\pi}\log^+|\frac{u_{\alpha j}'(re^{i\theta})}{u_{\alpha j}(re^{i\theta})}|^2d\theta\\\nonumber
&\leq& \oc(\log^+ T_f(r)+\log r),
\end{eqnarray}
where $C$ is some constant. Since $f$ is non-algebraic, from Theorem \ref{dense} we know that
	$$
	\lim\limits_{r\rightarrow\infty}\frac{T_f(r)}{\log r}=+\infty,
	$$
	thus we have
	$$
	\lim\limits_{r\rightarrow\infty}-\frac{1}{T_f(r)}\int_{0}^{2\pi}\log 
	|f'(re^{i\theta})|^2_\omega d\theta\geq0.
	$$
	By (\ref{tower}) we have
\begin{equation}\label{tau formula}
\lim\limits_{r\rightarrow 
	+\infty}\frac{1}{T_f(r)}\int_{1}^{r}\frac{dt}{t}\int_{B(t)}f_1^*\Theta_{h^*}(\oc_{
	X_1}(-1))\geq 	\lim\limits_{r\rightarrow\infty}-\frac{1}{T_f(r)}\int_{0}^{2\pi}\log 
|f'(re^{i\theta})|^2_\omega d\theta\geq0.
\end{equation}
From the definition of $\omega_1$, we have
$$
T_{f_1,\omega_1}(r)=\int_{1}^{r}\frac{dt}{t}\int_{B(t)}f_1^*\omega_1=T_{f_1,
\pi^*\omega}(r)+\delta\int_{1}^{r}\frac{dt}{t}\int_{B(t)}f_1^*\Theta_{h}(\oc_
{X_1}(1)).
$$ 
By $T_{f,\omega}(r)=T_{f_1,\pi^*\omega}(r)$ we get
$$
\lim\limits_{r\rightarrow 
\infty}\frac{T_{f_1,\pi^*\omega}(r)}{T_{f_1,\omega_1}(r)}\geq 1.
$$
By Remark \ref{semi} we can replace $\omega_1$ by $\pi^*\omega$ in the 
definition of Ahlfors current $T[f_1]$, and the equality
$T_{f_1,\pi^*\eta}(r)=T_{f,\eta}(r)$ for any $(1,1)$ form $\eta$ then yields
$$
\pi_*T[f_1]=T[f].
$$
\end{proof}
Similarly we have the following lemma:
\begin{lem}\label{bimero def}
	Let $p:Y\rightarrow X$ be a bimeromorphic morphism between K\"ahler manifolds $X$ and $Y$, obtained by a sequence of blow-ups with smooth centers.
	If $f:\cb\rightarrow X$ is an entire curve whose image is not contained in the critical value of $p$, then we can 
	lift $f$ as $\widetilde{f}:\mathbb{C}\rightarrow Y$ and define the Ahlfors 
	current $T[\widetilde{f}]$ with respect to $p^*\omega$ in such a way that 
	$p_*T[\widetilde{f}]=T[f]$.
\end{lem}
\begin{proof}
	We first fix a K\"ahler metric on $X$. Since $Y$ is obtained by a finite sequence of blow-ups, we can find a K\"ahler metric $\widetilde{\omega}$ on $Y$ defined by 
	$$
	\widetilde{\omega}=p^*\omega-\sum_i \epsilon_i \Theta_{h_i}(E_i),
	$$
	where all $E_i$ are irreducible divisors supported in the exceptional locus of $p$, $h_i$ is some smooth hermitian metric on $E_i$, and each $\epsilon_i$ is some positive real number.
	Since the image of $f$ is not contained in the critical value of $p$, the image of $\widetilde{f}$ is not contained in the exceptional locus of~$p$, and \emph{a fortiori} in any of the $E_i$. We claim that for all $E_i$ we have
\begin{equation}\label{not contain}
\langle T[f], \Theta_{h_i}(E_i)\rangle \geq 0.
\end{equation}
	Indeed, since the image of $\widetilde{f}$ is not contained in $E_i$, the counting function of $f$ with respect to $E_i$ is non-negative, and if we normalize the hermitian metric $h_i$, the proximity function is also non-negative. By Nevanlinna's First Main Theorem, we infer (\ref{not contain}).
\end{proof}
Therefore
$$
\lim\limits_{r\rightarrow 
	\infty}\frac{T_{\widetilde{f},p^*\omega}(r)}{T_{\widetilde{f},\widetilde{\omega}}(r)}\geq 1,
$$
and by Remark \ref{semi} we can define the Ahlfors 
current $T[\widetilde{f}]$ with respect to $p^*\omega$ in such a way that 
$p_*T[\widetilde{f}]=T[f]$.

\begin{rem}\label{sheaf}{\rm
For any coherent ideal sheaf $\jc$ whose zero 
scheme does not contain the image of $f:\cb\rightarrow X$, one can take a log resolution 
$p:\hat{X}\rightarrow X$ of $\jc$ with $p^*\jc=\oc_{\hat
X}(-D)$, and by Lemma \ref{bimero def} one can find a suitable sequence $(r_k)$ such that
$$
\langle T[\hat{f}], \Theta(D)\rangle =\lim\limits_{r_k\rightarrow 
\infty}\frac{T_{\hat{f},\Theta(D)}(r_k)}{T_{\hat{f},p^*\omega}(r_k)}=\lim\limits_
{r_k\rightarrow \infty}\frac{T_{\hat{f},\Theta(D)}(r_k)}{T_{f,\omega}(r_k)},
$$
where $\hat{f}$ is the lift of $f$ to $\hat{X}$ and $\Theta(D)$ is a curvature 
form of $\oc_{\hat{X}}(D)$ with respect to some smooth metric. By Theorem 
\ref{first}, we know that $\langle T[\hat{f}], \Theta(D)\rangle$ does not depend on the log 
resolution of $\jc$. We will denote this intersection number by $T(f,\jc)$.}
\end{rem}

\subsection{Intersection with the tangent bundle}
\begin{thm}(Tautological inequality)\label{tauto}
Let $f:(\cb,T_{\cb}) \rightarrow (X,V)$ be a transcendental entire curve in $X$, 
where $(X,V)$ is a smooth directed variety. We denote by $f_1:\cb\rightarrow P(V)$ the 
lift of $f$. Let $\omega$ be a hermitian metric on $X$, and consider the smooth metric induced by $h$ on the line bundle $\oc_{P(V)}(1)$. Then we have
$$
\langle T[f_1],\oc_{P(V)}(-1)\rangle=\langle T[f_1],\Theta_{h^*}(\oc_{P(V)}(-1))\rangle\geq 0.
$$
\end{thm}
\begin{proof}
By (\ref{tau formula}) above we have
$$
\lim\limits_{r\rightarrow 
+\infty}\frac{1}{T_{f,\omega}(r)}\int_{1}^{r}\frac{dt}{t}\int_{B(t)}f_1^*\Theta_{h^*}(\oc_{
P(V)}(-1))\geq 0.
$$
By Lemma \ref{high tower} we can take $\pi^*\omega$ as the semi-positive form used in the definition of the Ahlfors current of $T[f_1]$, where $\pi:P(V)\rightarrow X$ is 
the natural projection. The equality $T_{f,\omega}(r)=T_{f_1,\pi^*\omega}(r)$ then implies
$$\langle T[f_1],\oc_{P(V)}(-1)\rangle= \lim\limits_{r\rightarrow 
+\infty}\frac{1}{T_{f_1,\pi^*\omega}(r)}\int_{1}^{r}\frac{dt}{t}\int_{B(t)}f_1^*\Theta_{h^*}(\oc_{
P(V)}(-1))\geq 0.$$

\end{proof}

\begin{rem}\label{Euler}{\rm
Recall the following well known formula: if $C\subset X$ is a smooth algebraic 
curve and $\widetilde{C}\subset P(T_X)$ is its lift to $P(T_X)$, then 
$c_1(\oc_{P(T_X)}(-1))\cdot [\widetilde{C}]=\chi (C)$. Thus the above 
tautological inequality intuitively means that the ``Euler 
characteristic'' of the transcendental curve $f:\mathbb{C}\rightarrow X$
is non-negative.}
\end{rem}
\begin{thm}\label{formula tangent}
Let $X$ be a K\"ahler manifold endowed with a 1-dimensional foliation $\fc$, and
$f:\cb\rightarrow X$ be a transcendental entire curve tangent to $\fc$ such that its image is not contained in $\text{Sing}(\fc)$. Then we have
\begin{equation}\label{formula 1}
\langle T[f], c_1(T_\fc)\rangle+T(f,\jc_{\fc})=\langle 
T[f_1],\Theta_g(\oc_{X_1}(-1))\rangle\geq0,
\end{equation}
where $g$ is the smooth metric on $\oc_{X_1}(-1)$ induced by $\omega$.
\end{thm}

\begin{proof}
Let $(U_\alpha)$ be a partition of unit of $X$, and let us denote $\Omega_\alpha=f^{-1}(U_\alpha)$. Then we have
\begin{equation}\label{mero tan}
f'(\tau)=\lambda_\alpha(\tau) v_\alpha|_{f(\tau)}
\end{equation}
for some holomorphic function $\lambda_\alpha(\tau)$
on $\Omega_\alpha \setminus f^{-1}(\text{Sing}(\fc))$ (notice that $f^{-1}(\text{Sing}(\fc))\cap \Omega_\alpha$ is a set of isolated 
points since its image is not contained in $\text{Sing}(\fc)$). 
We denote by $\eta_j$ the multiplicities of $\lambda_{\alpha}(\tau)$ at $t_j$ (they may be 
negative if $f(t_j)\in \text{Sing}(\fc)$, however we have $\eta_j+\nu_j\geq 0$), where 
$\nu_j$ is the multiplicity of $f$ along $\jc_{\fc}$, i.e.
$$
dd^c\log |v_\alpha|^2_\omega\circ f(\tau)|_{B(r)}=\sum_{|t_j|< r} \nu_j\log|\tau-t_j|^2+\oc(1).
$$
Since $v_\alpha=g_{\alpha\beta}v_{\beta}$, then if $t_j\in \Omega_\alpha\cap \Omega_{\beta}$, $\lambda_\alpha$ and $\lambda_\beta$ have the same multiplicity at $t_j$, and thus $\eta_j$ does not depend on the partition on unity.
By the Lelong-Poincar\'e formula and (\ref{mero tan}) we have
\begin{eqnarray}\nonumber
dd^c\log|f'(\tau)|^2_\omega&=&\sum_{|t_j|< r}  \eta_j\delta_{t_j}+dd^c\log |v_\alpha|^2_\omega\circ f(\tau)\\\label{lelong}
&=&\sum_{|t_j|< r}  \eta_j\delta_{t_j}-f^*\Theta_{h_s},
\end{eqnarray}
where $h_s$ is the singular metric on $T_{\fc}$ whose local weight is $\varphi_\alpha=-\log |v_\alpha|^2_\omega$ (see (\ref{singular metric}) above). If we fix a smooth metric $h$ on $T_\fc$, then there exists a globally defined 
function $\varphi_s$ such that 
$$
h=h_se^{-\varphi_s},
$$  
and $\varphi_s$ is the characteristic function associated to 
the coherent sheaf $\jc_{\fc}$.
By (\ref{lelong}) we have
$$
dd^c\log|f'(\tau)|^2_\omega-f^*dd^c\varphi_s=\sum_{|t_j|< r}  \eta_j\delta_{t_j}-f^*\Theta_{h}(T_\fc).
$$
on $B(r)$. From (\ref{poin}) we know that
$$
\sum_{|t_j|<r}\mu_j\delta_{t_j}-f_1^*\Theta_{g}(\oc_{X_1}(-1))-f^*dd^c\varphi_s=\sum_{|t_j|< r}  \eta_j\delta_{t_j}-f^*\Theta_{h}(T_\fc),
$$
where $g$ is the smooth metric on $\oc_{X_1}(-1)$ induced by $\omega$, and $\mu_j$ is the multiplicity of $f'(\tau)$ at $t_j$. Therefore, as $\mu_j-\eta_j=\nu_j$, we have
$$
f^*\Theta_{h}(T_\fc)=-\sum_{|t_j|<r}\nu_j\delta_{t_j}+f_1^*\Theta_{g}(\oc_{X_1}(-1))+f^*dd^c\varphi_s
$$
on each $B(r)$. Then we have
\begin{eqnarray}\nonumber\label{sum}
\langle T_r[f], 
\Theta_h(T_\fc)\rangle&:=&\frac{1}{T_f(r)}\int_{1}^{r}\frac{dt}{t}\int_{
	B(t)} f^*\Theta_{h}(T_\fc)\\\nonumber
&=&\frac{1}{T_f(r)}\int_{1}^{r}\frac{dt}{t}\int_{
	B(t)}f_1^*\Theta_{g}(\oc_{X_1}(-1))
-\frac{1}{T_f(r)}\sum_{|t_j|<r}\nu_j\log\frac{r}{|t_j|}\\\nonumber
&+&\frac{1}{T_f(r)}
\frac{1}{2\pi}\int_{0}^{2\pi}\varphi_s\circ f(re^{i\theta}) d\theta-\frac{1}{T_f(r)}
\frac{1}{2\pi}\int_{0}^{2\pi}\varphi_s\circ f(e^{i\theta}) d\theta\nonumber\\
&=&\frac{1}{T_f(r)}\int_{1}^{r}\frac{dt}{t}\int_{
	B(t)}f_1^*\Theta_{g}(\oc_{X_1}(-1))\nonumber\\
&-&\frac{1}{T_f(r)}N_{f,\jc_\fc}(r)-\frac{1}{T_f(r)}m_{f,\jc_\fc}(r),
\end{eqnarray}
where  $N_{f,\jc_\fc}(r)$ and $m_{f,\jc_\fc}(r)$ are the counting and proximity function of $f$ with respect to $\jc_\fc$, and the second equality above is a consequence of the Jensen formula. 
Since $T[f]$ is the weak limit of the positive current $T_{r_k}[f]$ for some sequence $r_k\rightarrow \infty$, we have
$$
\langle T[f], \Theta_h(T_\fc)\rangle=\lim_{r_k\rightarrow+\infty} \langle T_{r_k}[f], 
\Theta_h(T_\fc)\rangle.
$$
From Theorem \ref{first}, Remark \ref{sheaf} and Lemma \ref{high tower} we conclude that 
$$
\langle T[f], \Theta_h(T_\fc)\rangle+T(f,\jc_\fc)=\langle 
T[f_1],\Theta_g(\oc_{X_1}(-1))\rangle\geq0.
$$
\end{proof}

In fact, $K_\fc\otimes\jc_{\fc}$ is the canonical sheaf $\mathcal{K}_\fc$ of 
$\fc$ defined in \cite{Dem14}, by using \emph{admissible metric}. We recall the following definition in \cite{Dem14}.
\begin{dfn}\label{canonical sheaf}
	We say that the canonical sheaf $\mathcal{K}_\fc$ is ``\emph{big}" if there exists some birational morphism $\nu_\alpha:(X_\alpha,\fc_\alpha)\rightarrow(X,\fc)$ of $(X,\fc)$ such that the invertible sheaf $\mu_\alpha^*\mathcal{K}_{\fc_\alpha}$ is big in the usual sense for some log resolution $\mu_\alpha:Y_\alpha\rightarrow X_\alpha$ of $\jc_{\fc_\alpha}$. The \emph{base locus} $Bs(\fc)$ of $\fc$ is defined to be
	$$
	Bs(\fc):=\bigcap_{\alpha}\nu_\alpha\circ\mu_\alpha \mathbf{B}_+(\mu_\alpha^*\mathcal{K}_{\fc_\alpha}),
	$$ 
	where $(X_\alpha,\fc_\alpha)$ varies among all the birational models with $\mu_\alpha^*\mathcal{K}_{\fc_\alpha}$ big.
\end{dfn}

The above Theorem \ref{formula tangent} then gives another proof of the following  \emph{Generalized Green-Griffiths conjecture} for rank $1$ foliations formulated in \cite{Dem12}.  Moreover we can specify more precisely the subvariety containing the images of all transcendental curves tangent to the foliation. The theorem is as follows:
\begin{cor}
	Let $(X,\fc)$ be a projective 1-foliated manifold and assume that $\mathcal{K}_\fc$ is big. If $f:\cb\rightarrow X$ is a transcendental entire curve tangent to $\fc$, its image must be contained in $\text{Sing}(\fc)\cup Bs(\fc)$. In particular, any entire curve tangent to $\fc$ must be algebraically degenerate.
\end{cor}
\begin{proof}
	Assume that the image of $f$ is not contained in $\text{Sing}(\fc)\cup Bs(\fc)$. We proceed by contradiction. By Definition \ref{canonical sheaf}, there exists a birational morphism $\nu_\alpha:(X_\alpha,\fc_\alpha)\rightarrow(X,\fc)$ such that the invertible sheaf $\mu_\alpha^*\mathcal{K}_{\fc_\alpha}$ is big in the usual sense, for some log resolution $\mu_\alpha:Y_\alpha\rightarrow X_\alpha$ of $\jc_{\fc_\alpha}$, and such that the image of $f_\alpha$ is not contained in $\mathbf{B}_+(\mu_\alpha^*\mathcal{K}_{\fc_\alpha})$, where $f_\alpha$ is the lift of $f$ to $Y_\alpha$. We denote by $\widetilde{f}_\alpha$ the lift of $f$ to $X_\alpha$. By Theorem \ref{nef} we have
	$$
	\langle T[f_\alpha], c_1(\mu_\alpha^*\mathcal{K}_{\fc_\alpha}) \rangle>0.
	$$
	By Remark \ref{sheaf} and the fact that ${\nu_\alpha}_*T[f_\alpha]=T[\widetilde{f}_\alpha]\rangle $, we get
	$$
	\langle T[\widetilde{f}_\alpha], K_{\fc_\alpha}\rangle- T(\widetilde{f}_\alpha,\jc_{\fc_\alpha})=\langle T[f_\alpha],\mu_\alpha^*\mathcal{K}_{\fc_\alpha}\rangle>0.
	$$
	However, since $f$ is transcendental, by Theorem \ref{formula tangent} we infer
	$$
	\langle T[\widetilde{f}_\alpha], T_{\fc_\alpha}\rangle+ T(\widetilde{f}_\alpha,\jc_{\fc_\alpha})\geq 0,
	$$
	and the contradiction is obtained by observing that $c_1(K_{\fc_\alpha})=-c_1(T_{\fc_\alpha})$.
\end{proof}

\begin{rem}{\rm
	In \cite{Den15} we have generalized the above theorem to any singular directed variety $(X,V)$ (without assuming $V$ to be involutive), by applying the Ahlfors-Schwarz Lemma. In the proof, the canonical sheaf plays a crucial role (and it arises in a natural way).}
\end{rem}

By a result due to Takayama (ref. \cite{Tak08}) we know that on a projective manifold $X$ of general type, every irreducible component of $\mathbf{B}_+(K_X)$ is uniruled. It is natural to ask the following analogous question:

\begin{problem}
	Let $(X,\fc)$ be a projective 1-foliated pair with $\mathcal{K}_\fc$ big. Is every irreducible component of $Bs(\fc)$ uniruled?
\end{problem}

From (\ref{formula 1}), we see that the positivity of $T[f]\cdot T_\fc$ is controlled by the error term $T(f,\jc_{\fc})$. This gives us the advantage that we only need to compute the coherent ideal sheaf $\jc_{\fc}$ instead of knowing the exact form of $\fc$ at the singularities.

\begin{rem}{\rm
If $X$ is a complex surface endowed with a foliation $\fc$, and $C$ is an 
algebraic curve, then by Proposition 2.3 in \cite{Bru04} we have the formula:
\begin{equation}
C\cdot T_\fc+Z(\fc,C)=\chi(C),
\end{equation}
where $Z(\fc,C)$ is the multiplicity of the singularities of $\fc$ along the 
curve $C$. By Remark \ref{Euler}, we know that $\langle T[f_1],\oc_{X_1}(-1)\rangle$ can be 
seen as the ``Euler characteristic" of $f$. Thus (\ref{formula 1}) is a 
transcendental version of the above formula, and we have the following 
metamathematical correspondences between any algebraic leaf $C$ and 
any transcendental leaf $f:\mathbb{C}\rightarrow X$:
\begin{eqnarray}\nonumber
Z(\fc,C)\sim T(f,\jc_{\fc}),\\ \nonumber
C\cdot T_\fc\sim  \langle T[f],c_1(T_\fc)\rangle,\\ \nonumber
\chi(C)\sim \langle T[f_1],\oc_{X_1}(-1)\rangle.
\end{eqnarray}
\vskip-\baselineskip\vskip-\belowdisplayskip\hfill\qed
\vskip\belowdisplayskip}
\end{rem}

We also need the following logarithmic version of Theorem \ref{formula tangent}:
\begin{thm}\label{log version}
Let $(X,\fc,D)$ be a K\"ahler 1-foliated triple, and let $f:\cb\rightarrow X$ be a transcendental entire curve  tangent to 
$\fc$ such that its image is not contained in $\text{Sing}(\fc)\cup |D|$, where $|D|$ is the support of $D$. Then we have
$$
\langle T[f], c_1(T_\fc)\rangle+T(f,\jc_{\fc,D})=\langle T[\bar{f}_1], 
\oc_{\bar{X}_1}(-1)\rangle \geq -\liminf\limits_{r\rightarrow 
\infty}\frac{N_{f, D}^{(1)}(r)}{T_f(r)}=:-N^{(1)}(f,D),
$$
where $\bar{f}_1$ is the lift of $f$ on $\bar{X}_1:=P(T_X\langle -\log D\rangle)$, and 
$N_{f,D}^{(1)}(r)$ is the \emph{truncated counting function} of $f$ with respect to $D$
defined by
$$
N_{f,D}^{(1)}(r):=\sum_{|t_j|<r,f(t_j)\in D}\log\frac{r}{|t_j|}.
$$
\end{thm}

\begin{proof}
Since the image of $f$ is not contained in $|D|$, the condition that $f(t_j)\in D$ implies $f(t_j)\in \text{Sing}(\fc)$.

We use the notation and concepts introduced in Section \ref{notions}. Let $(U_\alpha)$ be a partition of unity on~$X$. On each $U_\alpha$ we have
$$
v_\alpha=\sum_{j=1}^{k}z_ja_{\alpha}^j\frac{\d}{\d 
z_j}+\sum_{i=k+1}^{n}a_{\alpha}^i\frac{\d}{\d z_i}
$$
as the generator of $\fc$, where $z_1\cdots z_k=0$ is the local equation of $D$ in $U_\alpha$. The hermitian metric $\omega_{X,D}$ induces a singular metric 
$h_{s,D}$ on $T_\fc$ with local weight 
$$
\varphi_{\alpha,D}=-\log |v_\alpha|^2_{\omega_{X,D}}=-\log 
\sum_{i,j}a^i_\alpha\overline{a^j_\alpha}\omega_{i\overline{j}}.
$$
If $h=h_se^{-\varphi_{s,D}}$ is a smooth metric  on $T_\fc$, then  
$\varphi_{s,D}$ is the characteristic function associated with $\jc_{\fc,D}$. 

Since the image of $f$ is not contained in $\text{Sing}(\fc)\cup |D|$, on $U_\alpha$ we have
\begin{eqnarray}\label{log tangent}
\bar{f}'(\tau):=(\frac{f_1'}{f_1},\ldots,\frac{f_k'}{f_k},f'_{k+1},\ldots,
f'_n)=\lambda_\alpha(\tau)(a^1_\alpha(f),\ldots,a^n_\alpha(f)),
\end{eqnarray}
where $\lambda_\alpha(\tau)$ is the meromorphic functions with poles contained in 
$f^{-1}(\text{Sing}(\fc)\cup |D|)$. By (\ref{log tangent}) we know that $f(t_j)\in D$ implies $f(t_j)\in \text{Sing}(\fc)$. Indeed, if $f(t_j)\in D$, then $\lambda_\alpha$ has a pole of order at least 1 at $t_j$, and such a pole can only occur at a point of $\text{Sing}(\fc)$.

Observe that $\bar{f}'(\tau)$ can be seen as a \emph{meromorphic} 
section of  $\bar{f}_1^*\oc_{\bar{X}_1}(-1)$, where $\bar{f}_1(\tau)$ is the lift 
of $f$ to $\bar{X}_1$, and denote $\Omega_\alpha=f^{-1}(U_\alpha)$. Then on $\Omega_\alpha\cap B(r)$ we have
$$
dd^c\log 
|f'(\tau)|^2_{\omega_{X,D}}=\sum_{|t_j|<r,t_j\in \Omega_\alpha}\eta_j\delta_{t_j}
+f^*dd^c\log|v_\alpha|^2_{\omega_{X,D}},
$$
where $\eta_j$ is the vanishing order of $\lambda_\alpha(\tau)$ on $t_j$. Since $v_\alpha=g_{\alpha\beta}v_{\beta}$, we see that $\eta_j$ does not depend on the partition of unity, and thus on $B(r)$ we have
\begin{equation}\label{log 1}
dd^c\log 
|f'(\tau)|^2_{\omega_{X,D}}=\sum_{|t_j|<r}\eta_j\delta_{t_j}
-f^*\Theta_{h}(T_\fc)+f^*dd^c\varphi_{s,D}.
\end{equation}

On the other hand, since 
$\omega_{X,D}$ induces a natural smooth metric on $T_X\langle -\log D\rangle$, it also induces a smooth hermitian metric $\bar{h}_1$ on $\oc_{\bar{X}_1}(-1)$, 
thus
\begin{equation}\label{log 2}
dd^c \log |f'(\tau)|^2_{\omega_{X,D}}=dd^c\log 
|\bar{f}'(\tau)|^2_{\bar{h}_1}=\sum_{0<|t_j|<r}\mu_j\delta_{t_j}-\bar{f}^*_1\Theta_{\bar{h}_1}
(\oc_{\bar{X}_1}(-1))
\end{equation}
on $B(r)$, where $\mu_j$ is the vanishing order of $\bar{f}'(t)$. By (\ref{log tangent}) we know that $\mu_j=-1$ if and only if $f(t_j)\in |D|$, and otherwise $\mu_j\geq 0$.   Then by using the logarithmic derivative lemma again as in Lemma \ref{high tower}, we find
\begin{equation}\label{log tauto}
\langle T[\bar{f}_1], \oc_{\bar{X}_1}(-1)\rangle \geq 
-\liminf\limits_{r\rightarrow \infty}\frac{N_{f,D}^{(1)}(r)}{T_f(r)}.
\end{equation}
We can combine (\ref{log 1}) and (\ref{log 2}) together to obtain
$$
f^*\Theta_{h}(T_\fc)=-\sum_{0<|t_j|<r}\mu_j\delta_{t_j}+\sum_{|t_j|<r}\eta_j\delta_{t_j}+\bar{f}^*_1\Theta_{\bar{h}_1}
(\oc_{\bar{X}_1}(-1))
+f^*dd^c\varphi_{s,D}
$$
on $B(r)$. By arguments very similar to those in the proof of Theorem \ref{formula tangent}, we get
\begin{eqnarray}\nonumber
\langle T[f], \Theta_h(T_\fc)\rangle
=\langle T[\bar{f}_1], \oc_{\bar{X}_1}(-1)\rangle-T(f,\jc_{\fc,D}),
\end{eqnarray}
and the theorem follows from (\ref{log tauto}).
\end{proof}

\subsection{Siu's refined tautological inequality}

In \cite{Siu02} Y-T. Siu proved McQuillan's ``refined tautological inequality'' 
by applying the traditional function-theoretical formulation. We will give here 
an improvement of this result. First we begin with the following lemma due to Siu.

\begin{lem}\label{siu lemma}
	Let $U$ be an open neighborhood of $0$ in $\mathbf{C}^n$ and 
	$\pi:\widetilde{U}\rightarrow U$ be the blow-up at $0$. Then 
	$\pi^*(\oc_U(\Omega_U^1))\subset \ic_E\otimes(\Omega^1_{\widetilde{U}}\langle -\log 
	E\rangle)$, where $\ic_E$ is the ideal sheaf of the exceptional divisor~$E$.
\end{lem}

\begin{thm}\label{siu formula}
	Let $H$ be an ample line bundle on a projective manifold $X$ of dimension $n$. 
	Let $Z$ be a finite subset of $X$ and $f:\cb\rightarrow X$ be an entire curve. 
	Let $\sigma\in H^0(X,S^l\Omega_X\otimes (klH))$ be such that $f^*\sigma$ is not 
	identically zero on $\cb$. Let $W$ be the zero divisor of $\sigma$ in 
	$X_1:=P(T_X)$, and $\pi:Y\rightarrow X$ be the blow-up of $Z$ with 
	$E:=\pi^{-1}(Z)$. Then we have
	\begin{equation}\label{siu}
	\frac{1}{l}N_{f_1,W}(r)+T_{\hat{f},\Theta_E}(r)-N^{(1)}_{f,m_Z}(r)\leq 
	kT_{f,\Theta_H}(r)+\oc(\log T_{f,\Theta_H}(r)+\log r),
	\end{equation}
	where $N^{(1)}_{f,m_Z}(r)$ is the truncated counting function with respect to the 
	ideal $m_Z$, and $\Theta_H$ (resp. $\Theta_E$) is the curvature of $H$ (resp. 
	$\Theta_H$) with respect to some smooth metric $h_H$ (resp. $h_E$).
\end{thm}

\begin{rem}{\rm
	In \cite{Siu02} and \cite{McQ98}, a slightly weaker inequality is obtained comparatively to (\ref{siu}), with $m_{f,m_Z}(r)$ in place of 
	$T_{\hat{f},\Theta_E}(r)-N^{(1)}_{f,m_Z}(r)$.}
\end{rem}
\begin{proof}[Proof of Theorem \ref{siu formula}]
	Let $\tau=\pi^*\sigma$. By Lemma \ref{siu lemma}, $\tau$ is a holomorphic 
	section of $S^l(\Omega_{Y}(\log E))\otimes \pi^*(klH)$ over $Y$ and $\tau$ vanishes 
	to order at least $l$ on $E$. Let $s_E$ be the canonical section of $E$. If we divide $\tau$ by $s^{\otimes l}_E$, then 
	$\widetilde{\tau}:=\frac{\tau}{s^{\otimes l}_E}$ is a holomorphic section of 
	$S^l(\Omega_{Y}(\log E))\otimes \pi^*(klH)\otimes (-lE)$ over $Y$. We now prove 
	that 
	\begin{equation}\label{equal}
	\lVert\tau(\bar{\hat{f}}'(t))\rVert_{\pi^*h^{\otimes 
			kl}_H}=\lVert\sigma({f'(t)})\rVert_{h^{\otimes kl}_H},
	\end{equation}
	where $\bar{\hat{f}}'(t)$ is the derivative of $\hat{f}$ in $T_{\hat{X}}\langle -\log E\rangle$ (see Definition \ref{log tangent}). To make things simple we assume $l=1$. Let $p$ be a point in $Z$ and let $U$ be a
	small open set containing $p$ such that locally we have
	$$
	\sigma=\sum_{i=1}^{n}a_idz_i\otimes e^{\otimes k},
	$$
	where $e$ is the local section of $H$ and $p$ is the origin. The blow-up at $p$ 
	is the complex submanifold of $U\times \mathbb{P}^{n-1}$ defined by 
	$w_jz_k=w_kz_j$ for $1\leq j\neq k\leq n$, where $[w_1:\dots:w_n]$ are the 
	homogeneous coordinates of $\mathbb{P}^n$. In the affine coordinate chart $w_1\neq0$ we have the relation
	$$
	(z_1,z_1w_2,\ldots,z_1w_n)=(z_1,z_2,\ldots,z_n),
	$$
	thus
	$$\tau=z_1\left( (a_1+\sum_{i=2}^{n}a_iw_i)d\log 
	z_1+\sum_{i=2}^{n}a_idw_i\right)\otimes (\pi^*e)^{\otimes k},$$
	and $\hat{f}(t)=(f_1,\frac{f_2}{f_1},\ldots,\frac{f_n}{f_1})$ in the local 
	coordinate $(z_1,w_2,\ldots,w_n)$. Thus 
	$$\bar{\hat{f}}'(t):=\left(\frac{f'_1}{f_1},(\frac{f_2}{f_1})',\ldots,(\frac{f_n}{f_1}
	)'\right)$$
        with respect to the local section $(z_1\frac{\d}{\d z_1},\frac{\d}{\d 
		w_2},\ldots,\frac{\d}{\d w_n})$ of $T_{\hat{X}}\langle -\log E\rangle$. It is easy to 
	check the equality (\ref{equal}).

	By the logarithmic derivative lemma again we know that $\frac{1}{2\pi}\int_{0}^{2\pi}\log^+ 
	\lVert\tau(\bar{\hat{f}}'(re^{i\theta}))\rVert_{\pi^*h^{\otimes kl}_H}$ and \newline 
	$\frac{1}{2\pi}\int_{0}^{2\pi}\log^+ 
	\lVert\widetilde{\tau}(\bar{\hat{f}}'(re^{i\theta}))\rVert_{\pi^*h^{\otimes kl}_H\otimes 
		h^{*\otimes l}_{E}}$ are both of the order $\oc(\log T_{f,\Theta_H}(r)+\log 
	r)$. Using $\log x=\log^+x-\log^+\frac{1}{x}$ for any $x>0$, we obtain
	\begin{eqnarray}\nonumber
	2lm_{\hat{f},E}(r)&=&\frac{1}{2\pi}\int_{0}^{2\pi} \log^+\frac{1}{\lVert s^{\otimes l}_E\circ 
		\hat{f}(re^{i\theta})\rVert^2_{h_E}}\\\nonumber
	&\leq&\frac{1}{2\pi}\int_{0}^{2\pi} 
	\log^+\frac{1}{\lVert\tau(\bar{\hat{f}}'(re^{i\theta}))\rVert^2_{\pi^*h^{\otimes 
				kl}_H}}+\frac{1}{2\pi}\int_{0}^{2\pi}\log^+ 
	\lVert\widetilde{\tau}(\bar{\hat{f}}'(re^{i\theta}))\rVert^2_{\pi^*h^{\otimes kl}_H\otimes 
		h^{*\otimes l}_{E}}+\oc(1)\\\nonumber
	&=&-\frac{1}{2\pi}\int_{0}^{2\pi} \log\lVert\tau(\bar{\hat{f}}'(re^{i\theta}))\rVert^2_{\pi^*h^{\otimes 
			kl}_H}+\frac{1}{2\pi}\int_{0}^{2\pi} \log^+\lVert\tau(\bar{\hat{f}}'(re^{i\theta}))\rVert^2_{\pi^*h^{\otimes 
			kl}_H}\\\nonumber
	&+&\frac{1}{2\pi}\int_{0}^{2\pi}\log^+
	\lVert\widetilde{\tau}(\bar{\hat{f}}'(re^{i\theta}))\rVert^2_{\pi^*h^{\otimes kl}_H\otimes 
		h^{*\otimes l}_{E}}+\oc(1)\\\nonumber
	&=&-\frac{1}{2\pi}\int_{0}^{2\pi} \log\lVert\tau(\bar{\hat{f}}'(re^{i\theta}))\rVert^2_{\pi^*h^{\otimes 
			kl}_H}+\oc(\log T_{f,\Theta_H}(r)+\log r)\\\label{corres}
	&=&-\frac{1}{2\pi}\int_{0}^{2\pi} \log\lVert\sigma({f'(re^{i\theta})})\rVert^2_{h^{\otimes kl}_H}+\oc(\log 
	T_{f,\Theta_H}(r)+\log r),
	\end{eqnarray}
	where the last equality is due to equality (\ref{equal}). Observe that there is a 
	natural isomorphism betweem
	$H^0(X,S^l(\Omega_{X})\otimes (klH))$ and $H^0(X_1,\oc_{X_1}(l)\otimes 
	p^*(klH))$, where $X_1:=P(T_X)$ and $p$ is its projection to $X$. We denote by 
	$P_\sigma$ the corresponding section of $\sigma$ in 
	$H^0(X_1,\oc_{X_1}(l)\otimes p^*(klH))$, whose zero divisor is $W$. Then we have
	$$
	\lVert P_\sigma (f_1(t))\cdot (f'(t))^l\rVert_{p^*h^{\otimes 
			kl}_H}=\lVert\sigma (f'(t))\rVert_{h^{\otimes kl}_H},
	$$
	and thus on $B(r)$ we have
	$$
	N(\sigma (f'(t)),r)=N_{f_1,W}(r)+l\sum_{|t_j|<r}\mu_j\frac{r}{|t_j|},
	$$
	where $N(\sigma (f'(t)),r)$ is the counting function of $\sigma (f'(t))$ and 
	$\mu_j$ is the vanishing order of $f'(t)$ at $t_j$. Therefore, by applying the Jensen formula to the last term in (\ref{corres}) we obtain
\begin{equation}\label{estimate siu}
2lm_{\hat{f},E}(r)+2N(\sigma (f'(t)),r)\leq 2klT_{f,\Theta_H}(r)+\oc(\log 
T_{f,\Theta_H}(r)+\log r)). 
\end{equation}
	Since we have
	$$
	N^{(1)}_{f,m_Z}(r)+\sum_{|t_j|<r,f(t_j)\in Z}\mu_j\frac{r}{|t_j|}=N_{f,m_Z}(r)=N_{\hat{f},E}(r),
	$$
	then by applying Nevanlinna's First Main
	Theorem to $(\hat{f},E)$ we get
	$$
	T_{\hat{f},\Theta_E}(r)=N_{\hat{f},E}(r)+m_{\hat{f},E}(r)+\oc(1),
	$$
	and we can combine this with (\ref{estimate siu}) to obtain
	$$
	T_{\hat{f},\Theta_E}(r)-N^{(1)}_{f,m_Z}(r)+\frac{1}{l}N_{f_1,W}(r)\leq 
	kT_{f,\Theta_H}(r)+\oc(\log T_{f,\Theta_H}(r)+\log r)).
	$$
\end{proof}

Now we have the following refined tautological equality:
\begin{thm}
	Let $X$ be a K\"ahler manifold of dimension $n$ and $f:\cb\rightarrow X$ be 
	a transcendental entire curve. Then for any finite set $Z$ we have
	$$
	T[f_1]\cdot \oc_{X_1}(-1)\geq T(f,m_Z)-N^{(1)}_{f,m_Z}\geq m(f,m_Z):=\lim\limits_{r\rightarrow 
		\infty}\frac{m_{f, m_Z}(r)}{T_{f,\Theta_H}(r)}.
	$$
\end{thm}
\begin{proof}
	First we choose $k$ large enough, in such a way that $\oc_{X_1}(1)\otimes p^*(kH)$ is 
	ample over $X_1$. When we choose $l$ sufficient large, $\oc_{X_1}(l)\otimes 
	p^*(lkH)$ will be very ample over $X_1$. Hence there exists a section $\sigma\in 
	H^0(X_1, \oc_{X_1}(l)\otimes p^*(lkH))$ whose defect is zero, i.e.
	$$
	N(f_1,W):=\lim\limits_{r_k\rightarrow \infty}\frac{N_{f_1,W}(r_k)}{T_{f,\Theta_H}(r_k)}=\langle T[f_1],\oc_{X_1}(l)\otimes p^*(lkH)\rangle=\langle T[f_1],\oc_{X_1}(l)\rangle+kl,
	$$
	where $W$ is the zero divisor of $\sigma$, and where the last equality comes from $\langle T[f],\Theta_H\rangle=1$. Since $f$ is transcendental, by Theorem \ref{dense} we have
	$$
	\lim\limits_{r\rightarrow\infty}\frac{T_{f,\Theta_H}(r)}{\log r}=+\infty.
	$$
	We can thus divide both sides in (\ref{siu}) by $T_{f,\Theta_H}(r)$ and take $r\rightarrow \infty$ to obtain
	$$
	\frac{1}{l}N(f_1,W)+T(f,m_Z)-N^{(1)}_{f,m_Z}\leq k,
	$$
	and we obtain the formula in the theorem.
\end{proof}

\subsection{Intersection with the normal bundle}\label{normal section}

As an application for Theorem \ref{formula tangent}, we will study the 
intersection of $T[f]$ with the normal bundle; i.e.\ with $c_1(N_\fc)$. 
Before anything else, we begin with the following definition.

\begin{dfn}\label{reduced singularity}
Let $X$ be a K\"ahler manifold endowed with a foliation $\fc$ by curves. We say 
that $\fc$ has \emph{weakly reduced singularities} if
\begin{enumerate}
 \item For some log resolution $\pi:\hat{X}\rightarrow X$ of $\jc_{\fc}$, we 
have $T_{\hat{\fc}}=\pi^*T_\fc$, where $\hat{\fc}$ is the induced foliation of 
$\pi^*\fc$;
 \item the $L^2$ multiplier ideal sheaf $\ic(\jc_{\fc})$ of $\jc_{\fc}$ is equal 
to $\oc_X$, i.e.,  at each $p\in X$, assume that $v$ is the local generator of $\fc$ around $p$, then for any $f\in \oc_{X,p}$, we have
$$
\frac{|f|^2}{|v|^2_{\omega}}\in L^1_{\text{loc}},
$$
where $\omega$ is any smooth hermitian metric on $X$.
\end{enumerate}
\end{dfn}

With the previous definition, we have the following theorem.

\begin{thm}\label{negative normal}
Let $X$ be a projective manifold of dimension $n$ endowed with a foliation 
$\fc$ by curves with weakly reduced singularities. If $f:\cb\rightarrow X$ is a transcendental entire curve tangent to $\fc$, whose image is not contained in $\text{Sing}(\fc)$ and satisfies $\langle T[f], 
K_X\rangle>0$ (e.g. $K_X$ is big), then we have
$$
\langle T[\hat{f}], c_1(N_{\hat{\fc}})\rangle<0
$$
for some birational modification $(\hat{X},\hat{\fc})$ of $(X,\fc)$.
\end{thm}
\begin{proof}
From the standard short exact sequence 
$$
0\longrightarrow T_\fc\longrightarrow T_X\longrightarrow N_\fc\longrightarrow 0
$$
that holds outside of a codimension 2 subvariety, we find
$$
c_1(K_X)+ c_1(T_\fc)=-c_1(N_\fc).
$$
By the definition of multiplier ideal sheaves (ref. \cite{Laz04}) we have
$$
\ic(\jc_{\fc})=\pi_*(K_{\hat{X}/X}-D),
$$
where $\pi:\hat{X}\rightarrow X$ is a log resolution of $\jc_{\fc}$ such that 
$\pi^*\jc_{\fc}=\oc_{\hat{X}}(-D)$. We know that $\ic(\jc_{\fc})=\oc_X$ if 
and only if
$K_{\hat{X}/X}-D$ is effective. By the assumption that the image of $f$ is not contained in $\text{Sing}(\fc)$, i.e.\ in the zero scheme of $\jc_{\fc}$, we know that the image of $\hat{f}$ is not contained in the support of the exceptional divisor, and thus
$$
\langle T[\hat{f}],K_{\hat{X}/X}-D\rangle \geq 0.
$$
Therefore we have
\begin{eqnarray}\nonumber
\langle T[\hat{f}], K_{\hat{X}}+T_{\hat{\fc}}\rangle &=&\langle T[\hat{f}], 
\pi^*K_{X}+\pi^*T_{\fc}+K_{\hat{X}/X}\rangle\\\nonumber
&\geq&\langle T[\hat{f}], \pi^*K_{X}+\pi^*T_{\fc}+D\rangle\\\nonumber
&=& \langle T[f], K_{X}+T_{\fc}\rangle +T(f,\jc_{\fc}),
\end{eqnarray}
where the last equality follows from the fact that $\pi_*T[\hat{f}]=T[f]$ and $T(f,\jc_\fc)=\langle
 T[\hat{f}],D\rangle$ (see Remark \ref{sheaf}). By Theorem \ref{formula tangent} we have
$$
\langle T[f], T_{\fc}\rangle +T(f,\jc_{\fc})\geq0,
$$
thus 
$$
-\langle T[\hat{f}], c_1(N_{\hat{\fc}})\rangle=\langle T[\hat{f}], 
K_{\hat{X}}+T_{\hat{\fc}}\rangle\geq \langle T[f],K_X\rangle >0.
$$
The theorem is proved.
\end{proof}

By the reduction theorem of singularities for surface foliations due to Seidenberg, for 
any pair $(X_0,\fc_0)$ there exists a finite sequence of blow-ups such that the 
induced pair $(X,\fc)$ has singularities of one of the following two types:
\begin{itemize}
\item a non-degenerate singular point $x_0$, in the sense that
$$
\log\frac{|a_1|^2+|a_2|^2}{|z_1|^2+|z_2|^2}=\oc(1).
$$
\item a degenerate singular point $x_1$ ``of type $k$'', such that
$$
\log\frac{|a_1|^2+|a_2|^2}{|z_1|^2+|z_2|^{2k}}=\oc(1), 
$$
where $k\geq2$ (we also call this a ``type $k$'' singularity).
\end{itemize}
Moreover
\begin{itemize}
\item Any singular point of $\fc$ is either non-degenerate or degenerate
of type $k$.
\item For any blow-up $\pi:\hat{X}\rightarrow X$ of a point on $X$, we have 
$\pi^*T_\fc=T_{\hat{\fc}}$ for  the induced foliation $\hat{\fc}$.
\item For any blow-up of the non-degenerate point $x_0$, $\hat{\fc}$ has two 
singularities on the exceptional divisor and both of them are non-degenerate.
\item For the blow-up of a degenerate point $x_1$ of type $k$, on the 
exceptional divisor $\hat{\fc}$ has a non-degenerate singular point and a 
degenerate one with the same type as $x_1$.
\end{itemize}

By the property above it is easy to show that $\ic(\jc_{\fc})= \oc_{X}$ if 
$\fc$ is reduced, thus $\fc$ is weakly reduced in the sense of Definition 
\ref{reduced singularity}. Then we can get another proof of McQuillan's theorem 
without using his ``Diophantine approximation analysis'' (still, we will make use of ``Diophantine approximation" in the next section):
\begin{thm}
Let $X$ be a complex surface endowed with a foliation $\fc$. If $\langle T[f], 
K_X\rangle>0$ (e.g.\ if $K_X$ is big), then any entire curve tangent to $\fc$ is algebraically degenerate.
\end{thm}
\begin{proof}
Assume that we have a Zariski dense entire curve $f:\cb\rightarrow X$ tangent 
to $\fc$. We proceed by contradiction.

By Seidenberg's theorem there is a sequence of blow-ups 
$\pi:\widetilde{X}\rightarrow X$ such that the singularities of 
$\widetilde{\fc}=\pi^{-1}\fc$ are reduced, and the lift $\widetilde{f}$ of $f$ 
to $\widetilde{X}$ is still Zariski dense. Thus 
$\widetilde{\jc}$ is weakly reduced and by Theorem \ref{negative normal} we have
$$
\langle T[\hat{f}], c_1(N_{\hat{\fc}})\rangle<0
$$
for some birational pair $(\hat{X},\hat{\fc})$ which is obtained by resolving 
the ideal $\jc_{\widetilde{\fc}}$.
However, Theorem \ref{brunella} tells us that 
$$
\langle T[\hat{f}], c_1(N_{\hat{\fc}})\rangle\geq0
$$
and we get a contradiction.
\end{proof}

\subsection{A generalization of McQuillan's theorem}

Thanks to the above theorems, we can obtain certain generalizations of McQuillan's theorem to higher dimensional 
manifolds, under some assumptions for the foliations. We start with some relevant definitions and properties (and refer to \cite{CCS97} and \cite{Tom97} for further details).

\begin{dfn}
Let $\fc$ be a foliation by curves on a $n$-dimensional
complex manifold $X$. We say that $p_0\in \text{Sing}(\fc)$ is an \emph{absolutely
isolated singularity} (A.I.S.) of $\fc$ if and only if the following properties
are satisfied:
\begin{enumerate}
\item $p_0$ is an isolated singularity,
\item If we consider an arbitrary sequence of blowing-up's 
$$
(X,\fc)\xleftarrow{\pi_1} 
(X_1,\fc_1)\xleftarrow{\pi_2}\cdots\xleftarrow{\pi_{n}} (X_n,\fc_n)
$$
where the center of each blow-up $\pi_i$ is a singular point $p_{i-1}\in 
\text{Sing}\fc_{i-1}$, then $\# \text{Sing}(\fc_n)<+\infty$.
\end{enumerate}
\end{dfn}

Since locally the foliation $\fc$ is generated by a holomorphic vector field 
$v=\sum_{i=1}^{n}a_i\frac{\d}{\d z_i}$, we can define the \emph{algebraic 
multiplicity} $m_p(\fc)$ of $\fc$ at $p$ to be the minimum of the vanishing
orders $\text{ord}_p(a_i)$.
Recall that the linear part of $\fc$ at $p$ is defined by
$$
\mathcal{L}_v:m_p/m^2_p\rightarrow m_p/m^2_p.
$$
A singular point $p\in \text{Sing}\fc$ is called \emph{reduced} if $m_p(\fc)$ = 1 and 
the linear part of $\fc$ at $p$ has at least
one non-zero eigenvalue.   

We shall say that $p\in \text{Sing}\fc$ is a \emph{non-dicritical singularity} 
of $\fc$ if $\pi^{-1}(p)$ is invariant
by $\fc$, where $\pi$ is the blow-up of $p$. Otherwise $p$ is called a 
\emph{dicritical singularity}. 

Let $(X,\fc,D)$ be a 1-foliated triple. We assume that all singularities of $\fc$ lie on $D$ (this can be achieved after we take a log resolution of $\jc_{\fc}$). Fix a point $p\in \text{Sing}(\fc)$ and denote by $e=e(D,p)$ the number of irreducible components of $D$ through $p$. Since $\text{Sing}(\fc)\subset D$, we have $e\geq 1$. Then the vector fields which generate $\fc$ are given by
$$
v=\sum_{j=1}^{e}z_ja_j\frac{\d}{\d z_j}+\sum_{i=e+1}^{n}a_i\frac{\d}{\d z_i},
$$
where $z_1z_2\ldots z_e=0$ is the local equation of $D$ at $p$.

\begin{dfn}\label{simple singularity}
	Assume that $e=1$. Then $p$ is a \emph{simple point} iff one of the following two possibilities occurs:
	\begin{enumerate}[(A)]
		\item $a_1(0)=0$, the curve $(z_2=\ldots=z_n=0)$ is invariant by $\fc$ (up to an adequate formal choice of coordinates) and the linear part $\mathcal{L}_{v|_D}$ of $v|_D$ is of rank $n-1$.
		\item $a_1(0)=\lambda\neq 0$, the multiplicity of the eigenvalue $\lambda$ is one and if $\mu$ is another eigenvalue of the linear part of $\mathcal{L}_v$, then $\frac{\mu}{\lambda}\notin \mathbb{Q}_+$.
	\end{enumerate}
	
	Assume that $e\geq 2$. Then $p$ is a \emph{simple corner} iff (up to a reordering of $(z_1,\ldots,z_n)$), we have $a_1(0)=\lambda\neq 0, a_2(0)=\mu$ and $\frac{\mu}{\lambda}\notin \mathbb{Q}_+$.
	
	We say that $p$ is a \emph{simple singularity} iff it is a simple point or a simple corner.
\end{dfn}

The simple singularities are ``stable" under blowing-up:

\begin{proposition}\label{stable simple}
	Assume that $p$ is a simple singularity of the 1-foliated triple $(X,\fc,D)$. Let $\mu:\widetilde{X}\rightarrow X$ be the blow-up of $X$ with the center $p$, $\widetilde{D}:=\mu^{-1}(D\cup \{p\})$ and $\widetilde{\fc}$ be the induced foliation. Then:
	\begin{enumerate}[(a)]
		\item Each irreducible component of $\widetilde{D}$ is invariant by $\widetilde{\fc}$.
		\item  If $p'\in \widetilde{\fc}_s\cap \mu^{-1}(p)$, then $p'$ is also a simple singularity of $\widetilde{\fc}$ with respect to the induced 1-foliated triple $(\widetilde{X},\widetilde{\fc},\widetilde{D})$.  More precisely:
		\begin{enumerate}[(b-1)]
			\item if $p$ is a simple point, there is exactly one simple point $p'\in \widetilde{\fc}\cap \mu^{-1}(p)$. The other points in $\widetilde{\fc}\cap \mu^{-1}(p)$ are simple corners. Moreover, $p$ and $p'$ have the same type $(A)$ or $(B)$ of Definition \ref{simple singularity}.
			\item If $p$ is a simple corner, then all points in $\text{Sing}(\fc)\cap \mu^{-1}(p)$ are simple corners.
		\end{enumerate}
	\end{enumerate} 
\end{proposition}

In \cite{Tom97} and \cite{CCS97}, the following resolution theorem for absolutely isolated singularities has been proved.

\begin{thm}\label{resolve ais}
	Let $\fc$ be a foliation by curves on the $n$-dimensional complex manifold $X$, 
	such that the singularities $\text{Sing}(\fc)$ of the foliation $\fc$ is a set of absolutely isolated singularities. Then 
	there exists a finite sequence of blow-up's 
	$$
	(X,\fc)\xleftarrow{\pi_1} 
	(X_1,\fc_1)\xleftarrow{\pi_2}\cdots\xleftarrow{\pi_{n}} (X_n,\fc_n)
	$$
	satisfying the following property:
	\begin{enumerate}
		\item the center of each blow-up $\pi_i$ is a singular point $p_{i-1}\in \text{Sing}(\fc)$.
		\item $(X_n,\fc_n,D_n)$ is a 1-foliated triple with only simple singularities (which are, of course absolutely isolated singularities).
		\item All the singularities of $\fc_n$ are non-dicritical.
	\end{enumerate} 
\end{thm}

In \cite{Tom97}, the author also gave a classification of the linear part $\mathcal{L}_v:m_p/m^2_p\rightarrow m_p/m^2_p$ of reduced, non-dicritical A.I.S.

\begin{proposition}\label{tome linear}
	Let $p$ be a reduced, non-dicritical A.I.S of the foliation $\fc$, then its linear part can be written 
$$\mathcal{L}_v=\text{diag}[M(\lambda_1),\ldots,M(\lambda_s)]$$ 
	where $M(\lambda_k)=\lambda_kI_{r_k}+R_{r_k}(1)$ for  $1\leq k\leq s$. (Here $I_{r_k}$ is the identity matrix of $\cb^{r_k \times r_k}$ and $R_{r_k}(1)\in \cb^{r_k \times r_k}$ is the upper triangle matrix of order $1$).
	
	Moreover, if we denote by $\widetilde{\fc}$ the induced foliation of $\fc$ after the blow-up at $p$, then we have
	$$
	\text{Sing}(\widetilde{\fc})\cap \mu^{-1}\{p\}=\{\widetilde{p}_1,\ldots,\widetilde{p}_s\}
	$$
	where $\widetilde{p}_l$ is the zero at the charts $\widetilde{z}_j=z_j, \widetilde{z}_i=z_i/z_j$ for $i=1,\ldots,n$ and $j=r_1+\ldots+r_{l-1}+1$.
\end{proposition}

\begin{dfn}
	Let $\fc$ be a foliation by curves defined on some open domain $U\subset \cb^n$. A \emph{separatrix} of the singular holomorphic foliation $\fc$ at the point $p\in \text{Sing}(\fc)$ is a local leaf $L\subset (U,p)\setminus \text{Sing}(\fc)$ whose closure $L\cup p$ is a germ of analytic curve.
\end{dfn}

We can summarize Proposition \ref{stable simple} and Proposition \ref{tome linear} to state a proposition in the following final form:
\begin{proposition}\label{summary ais}
	Let $(X,\fc,D)$ be a 1-foliated triple. Let $p$ be an absolutely isolated and simple singularity of $\fc$ and let $w_1w_2\ldots w_e=0$ be the local equation of $D$ at $p$. Then we can find a new coordinate system $(z_1,\ldots,z_n)$ such that the linear part $\mathcal{L}_v$ of a generator of $\fc$ at $p$ can be written in the following Jordan form:
	$$
\mathcal{L}_v=\sum_{i=1}^{s}\lambda_i z_i\frac{\d}{\d z_i}+\sum_{j=1}^{k}\sum_{p=s+r_1+\ldots+r_{j-1}+1}^{s+r_1+\ldots+r_j-1}(\lambda_{s+j}z_p+z_{p+1})\frac{\d}{\d z_p}+\lambda_{s+j}z_{s+r_1+\ldots+r_j}\frac{\d}{\d z_{s+r_1+\ldots+r_j}}
	$$
	where $r_j\geq 0$ are the sizes of the Jordan blocks, $z_1 z_2\ldots z_e=0$ (of course, $e\leq s$) is the local equation of $D$ at $p$. $\lambda_j\neq 0$ for $j=2,\ldots,s+k$; $\frac{\lambda_i}{\lambda_j}\notin \mathbb{Q}_+$ for $i\neq j,i,j=1,\ldots,s+k,j\neq 1$. Here, the local generator of $\fc$ is
$$
	v=\sum_{j=1}^{e}z_ja_j\frac{\d}{\d z_j}+\sum_{i=e+1}^{n}a_i\frac{\d}{\d z_i},
	$$
and $a_j(0)=\lambda_j$ for $j=1,\ldots,e$.	
If we denote by $\widetilde{\fc}$ the foliation induced by $\fc$ after taking the blow-up at $p$, then we have
		$$
		\text{Sing}(\widetilde{\fc})\cap E=\{\widetilde{p}_1,\ldots,\widetilde{p}_{s+k}\}
		$$
		where $\widetilde{p}_l$ is the zero at the charts $\widetilde{z}_j=z_j, \widetilde{z}_i=z_i/z_j$ for $i=1,\ldots,n$ and $j=l$ if $l\leq s$ or $j=s+r_1+\ldots+r_{l-1}+1$ if $l=s+m$; $E$ is the exceptional divisor of the blow-up. Moreover, at each $\widetilde{p}_i$ the linear part $\mathcal{L}_{\widetilde{v}}$ of a generator of $\widetilde{\fc}$ is similar to that $\mathcal{L}_v$, since simple singularities are stable under blowing-up. 
\end{proposition}

As a consequence of Proposition \ref{summary ais} we get the following theorem:
\begin{thm}\label{separatrix}
	With the above notation, let $p$ be an absolutely isolated and simple singularity of $\fc$ and let $z_1z_2\ldots z_e=0$ be the local equation of $D$ at $p$. Then \begin{enumerate}[(a)]
		\item if $p$ is a simple corner (i.e. $e\geq 2$),  each separatrix of $\fc$ at $p$ must be contained in $D$.
		\item Let $(\widetilde{X},\widetilde{\fc},\widetilde{D})$ be the induced 1-foliated triple by the blow-up at $p$ with the exceptional divisor $E$, and $C$ be a separatrix at $p$ which is not contained in $D$. If $p$ is a simple point (i.e. $e=1$), then the lift $\widetilde{C}$ of each separatrix $C$ to $\widetilde{X}$ intesects with $E$ at $\widetilde{p}_1$, which is still a simple point with respect to $(\widetilde{X},\widetilde{\fc},\widetilde{D})$.
	\end{enumerate}
\end{thm}

\begin{proof}
 Assume that we have a separatrix $C$ of $\fc$ at $p$ which is not contained in $D$. We take a local parametrization $f:(\cb,0)\rightarrow (C,p)$ for this separatrix, then in the local coordinate system $(z_1,\ldots,z_n)$ introduced in Proposition \ref{summary ais},  we have
	$$
	\left( f'_1(t),\ldots, f'_n(t)\right)=\eta(t)\cdot \left( f_1(t)a_1(f),\ldots, f_e(t)a_e(f),a_{e+1}(f),\ldots, a_n(f) \right)
	$$
	for some meromorphic function $\eta(t)$ whose poles only appear at $0$.
By the assumption that $C$ is not contained in $D$, $f_i(t)$ is not identically equal to zero for $i=1,\ldots,e$. We denote by $\nu_i$ the vanishing order of $f_i(t)$ at 0 for $i=1\ldots,n$, which are all positive integers.

If $p$ is a simple corner, then $e\geq2$ and $\lambda_2=a_2(0)\neq 0$, and we have 
$$
\eta(t)a_2(t)=\frac{f'_2(t)}{f_2(t)}.
$$
This implies that the order of pole of $\eta(t)$ at 0 must be 1. If we denote by $\eta(t)=\frac{b(t)}{t}$ with $b(t)$ some germ of holomorphic function satisfying $b(0)\neq 0$, then
$$
b(0)\cdot \lambda_2=\nu_2.
$$
Similarly we have 
$$
b(0)\cdot \lambda_1=\nu_1>0,
$$
thus $\frac{\lambda_1}{\lambda_2}=\frac{\nu_1}{\nu_2}\in \mathbb{Q}_+$, which is a contradiction. Therefore any separatrix at the simple corner must be contained in $D$, and $p$ only can be a simple point (i.e. $e=1$).

Suppose now that $p$ is a simple point. Since $\widetilde{C}\cap E\in \text{Sing}(\widetilde{\fc})\cap E= \{\widetilde{p}_1,\ldots,\widetilde{p}_{s+k}\}$, we can assume that $\widetilde{C}\cap E=\{\widetilde{p}_l\}$ for some $l\geq2$. Since $\widetilde{p}_l$ is the zero at the charts $\widetilde{z}_j=z_j, \widetilde{z}_i=z_i/z_j$ for $i=1,\ldots,n$ and $j=l$ if $l\leq s$ or $j=s+r_1+\ldots+r_{m-1}+1$ if $l=s+m$, the lift $\widetilde{f}(t)$ of $f(t)$ to $\widetilde{X}$ is
$$
\left(\widetilde{f}_1(t)\ldots, \widetilde{f}_n(t)\right)=\left( \frac{f_1(t)}{f_{c_l}(t)},\ldots, \frac{f_{c_l-1}(t)}{f_{c_l}(t)},f_{c_l}(t),\frac{f_{c_l+1}(t)}{f_{c_l}(t)},\ldots, \frac{f_{n}(t)}{f_{c_l}(t)}\right),
$$
where
\begin{equation}
c_l=\left\{
\begin{aligned}
&l &\text{if} \quad l\leq s,\\
&j=s+r_1+\ldots+r_{m-1}+1 &\text{if} \quad l=s+m,
\end{aligned}
\right.\nonumber
\end{equation}
and it satisfies
$$
\widetilde{f}(0)=0.
$$
Hence we have
$$
\nu_i>\nu_{c_l} 
$$
for $i\neq c_l$. Therefore from the linear part of $a_{c_l}$ we have
\begin{equation}\nonumber
a_{c_l}(f(t))=\lambda_{c_l}f_{c_l}(t)+ \text{higher order term},
\end{equation}
where the higher order term is with respect to $t$. Thus we can denote by $\eta(t)=\frac{b(t)}{t}$ with $b(0)\neq 0$, and 
$$
b(0)\cdot \lambda_{c_l}=\nu_{c_l},
$$
which implies 
$$
b(0)\cdot \lambda_1=\nu_1>0,
$$
thus $\frac{\lambda_1}{\lambda_{c_l}}=\frac{\nu_1}{\nu_{c_l}}\in \mathbb{Q}_+$. This is a contradiction, and claim (b) of the theorem is proved.
\end{proof}

Thanks to the previous theorem, we can prove the following result, which can be seen as a generalization of McQuillan's ``Diaphantine approximation".
\begin{thm}\label{diophantine}
Let $(X,\fc,D)$ be a 1-foliated triple with absolutely isolated singularities, $H$ an ample divisor over $X$, and $p_0$ a simple point of $\fc$.  We 
take a sequence of blow-up's 
$$
(X,\fc,D)\xleftarrow{\pi_1} 
(X_1,\fc_1,D_1)\xleftarrow{\pi_2}\cdots\xleftarrow{\pi_{k}} (X_k,\fc_k,D_k)\xleftarrow{\pi_{k+1}}\cdots 
$$
such that the center of each blow-up $\pi_i$ is the simple point $p_{i-1}\in \text{Sing}({\fc_{i-1}})\cap E_{i-1}$, where $E_{i-1}$ is the exceptional divisor of the blow-up $\pi_{i-1}$ (from Proposition \ref{stable simple} we know that $p_{i-1}$ exists and is unique).
Then there exists a positive integer $\alpha>0$ such that $ \lceil k^{\frac{1}{n}}\rceil\alpha \mu_k^* H- kE_k$ is effective in $X_k$ for any $k>0$, where $\mu_k:=\pi_k\circ \ldots \circ \pi_1 $.
\end{thm}
\begin{proof}
	With the above notation, we can assume that $p_1$ is the origin of the charts $z^{(1)}_1=z_1, z^{(1)}_jz^{(1)}_1=z_j$ for $j\neq 1$, where $z_1=0$ is the local equation of $D$ at $p_0$. Inductively  $p_k$ is the origin of the charts $z^{(k)}_1=z^{(k-1)}_1, z^{(k)}_jz^{(k)}_1=z^{(k-1)}_j$ for $j\neq 1$, and $z^{(k)}_1=0$ is the local equation of $E_k$. Thus we have
	$$
	\mu_k(z^{(k)}_1,\ldots,z^{(k)}_n)=(z^{(k)}_1,z^{(k)k}_1z^{(k)}_2,\ldots,z^{(k)k}_1z^{(k)}_n),
	$$
	which implies
	$$
	\oc_{X_k}(-kE_k)\subset \mu^{*}_km_{(z^{k}_1,z_2,\ldots,z_n)},
	$$
	where $m_{(z^{k}_1,z_2,\ldots,z_n)}\subset \oc_{X,p}$ is the ideal generated by $(z^{k}_1,z_2,\ldots,z_n)$ at $p$.
	From the exact sequence
	$$
	0\longrightarrow \oc_X(\lceil k^{\frac{1}{n}}\rceil\alpha H)\cdot m_{(z^{k}_1,z_2,\ldots,z_n)}\longrightarrow \oc_X(\lceil k^{\frac{1}{n}}\rceil\alpha H) \longrightarrow \oc_{X,p}/m_{(z^{k}_1,z_2,\ldots,z_n)}\longrightarrow 0,
	$$
	and the Riemann-Roch theorem we can choose a fixed positive integer $\alpha>0$ such that
	$$
	H^0(X,\oc_X(\lceil k^{\frac{1}{n}}\rceil\alpha H)\cdot m_{(z^{k}_1,z_2,\ldots,z_n)})\neq 0
	$$
	for any $k>0$. If we choose a divisor $R_k\in	H^0(X,\oc_X(\lceil k^{\frac{1}{n}}\rceil\alpha H)\cdot m_{(z^{k}_1,z_2,\ldots,z_n)})$, then 
	$\mu^*_kR_k-kE_k$ is effective.
\end{proof}

\begin{thm}
	Let $\fc$ be a foliation by curves on a $n$-dimensional complex manifold $X$, 
	such that the singularities $\text{Sing}(\fc)$ of the foliation $\fc$ is a set of absolutely isolated singularities. If $f:\cb\rightarrow X$ is an Zariski dense entire curve which is tangent to $\fc$, then we can blow-up $X$ finitely many times to get a new birational model $(\widetilde{X},\widetilde{\fc})$ such that
	$$
	T[\widetilde{f}]\cdot T_{\widetilde{\fc}}=0.
	$$
\end{thm}
\begin{proof}
	First we assume that all singularities of $\fc$ are simple and non-dicritical, and that the linear forms of the generator of $\fc$ are like those described in Proposition \ref{summary ais}. Then, by Definition \ref{simple singularity}, we know that the coherent ideal sheaf $\jc_{\fc,D}$ is not trivial at $p$ if and only if $p$ is a simple point of type $(A)$. We denote by $\eta_p$ the least  integer $k$ such that $\frac{\d^k a_1}{\d z^k_1}(0)\neq 0$ (if $p$ is a simple point of type $(B)$ , let $\eta_p=0$). Without loss of generality we can assume that $p$ is the unique simple point among all singularities. 
	
	If we take a sequence of blow-ups as described in Theorem \ref{diophantine}, by Proposition \ref{stable simple} we know that on each $X_k$, $p_k$ is the unique simple point, and it is easy to verify that
	$$
	\eta_p=\eta_{p_1}=\cdots=\eta_{p_k}=\cdots.
	$$
	If $\eta_p=0$, then $\jc_{\fc_{k,D_k}}=\oc_{X_k}$ for any $k$; now we assume that $\eta_p>0$. Fix a $k$, repeat the blow-up procedure $\eta_p$ times, and resolve the coherent ideal sheaf $\jc_{\fc_{k,D_k}}$, i.e.,
	\begin{equation}\label{resolve}
		\mu^*_{k+\eta_p,k+1}\jc_{\fc_k,D_k}=\oc_{X_{k+\eta_p}}(-\widetilde{E}_{k+1}-\ldots-\widetilde{E}_{k+\eta_p-1}-E_{k+\eta_p}),
	\end{equation}
	where $\mu_{k+\eta_p,k+1}=\pi_{k+\eta_p}\circ\cdots\circ\pi_{k+1}$, and $\widetilde{E}_{i}$ is the strict transform of $E_{i}$ under $\pi_{k+\eta_p}\circ \cdots \circ \pi_{i+1}:X_{k+\eta_p}\rightarrow X_{i}$ (by Proposition \ref{summary ais} we have an explicit linear part of the local generator of $\fc_k$ at the simple point, it is easy to examine this).
	
	From Theorem \ref{diophantine}  we know that $\lceil k^{\frac{1}{n}}\rceil \mu_k^* H- kE_k$ is effective for some ample divisor $H$ and any $k>0$. Thus we have
	$$
	T[f_k]\cdot (\mu_k^* H- \frac{k}{\lceil k^{\frac{1}{n}}\rceil}E_k)\geq 0.
	$$
	By the equality ${\mu_k}_*T[f_k]=T[f]$  we get
	$$
	T[f_k]\cdot E_k\rightarrow 0
	$$
	as $k\rightarrow \infty$.
	By Theorem \ref{log version}
	$$
	\langle T[f_k], c_1(T_{\fc_k})\rangle+T(f_k,\jc_{\fc_k,D_k})\geq -N^{(1)}(f_k,\text{Sing}(\fc_k)\cap D_k),
	$$
	where $N^{(1)}(f_k,\text{Sing}(\fc_k)\cap D_k)$ is the truncated counting function. By Theorem \ref{separatrix} (a) we know that the image of $f$ contains no simple corners, otherwise $f(\cb)$ would be contained in $D$; therefore simple corners do not contribute to $N^{(1)}$. If $p\in f(\cb)$ for the unique simple point $p$, then Theorem \ref{separatrix} (b) tells us that $f_k(\cb)\cap \text{Sing}(\fc_k)=\{p_k\}$, and thus
	$$
	N^{(1)}(f_k,\text{Sing}(\fc_k)\cap D_k)\leq T[f_{k+1}]\cdot {E_{k+1}}.
	$$
	From (\ref{resolve}) we obtain
	$$
	T(f_k,\jc_{\fc_k,D_k})=\sum_{j=1}^{\eta_p}T[f_{k+j}]\cdot E_{k+j},
	$$
	and thus
	$$
		\langle T[f_k], c_1(T_{\fc_k})\rangle\geq -2T[f_{k+1}]\cdot {E_{k+1}}- \sum_{j=2}^{\eta_p}T[f_{k+j}]\cdot E_{k+j}.
	$$
	The right hand of the preceding inequality tends to 0 as $k\rightarrow \infty$.
	Since all singularities are non-dicritical, we have $\mu^*_k(T_\fc)=T_{\fc_k}$ and
	$$
	T[f_k]\cdot c_1(T_{\fc_k})=T[f]\cdot c_1(T_\fc),
	$$
	thus $T[f]\cdot c_1(T_\fc)\geq 0$.
	
	Now we come to the general case: there are some singularities which may not be simple singularities. By Theorem \ref{resolve ais} we can take a finite sequence of blow-ups with centers only at singularities, to get a new 1-foliated triple  $(\widetilde{X},\widetilde{\fc},\widetilde{D})$ with simple singularities which are all non-dicritical. By the proof above we have
	$$
	T[\widetilde{f}]\cdot c_1(T_{\widetilde{\fc}})\geq0.
	$$
	
	However, a theorem of Brunella implies that $K_{\widetilde{\fc}}$ is pseudo-effective, since $\widetilde{\fc}$ contains a transcendental leaf. Therefore
	$$
	T[\widetilde{f}]\cdot c_1(K_{\widetilde{\fc}})\geq0.
	$$
	If we combine the two inequalities, we obtain
	$$
	T[\widetilde{f}]\cdot c_1(T_{\widetilde{\fc}})=0.
	$$
\end{proof}

\section{Towards the Green-Griffiths conjecture}

In order to pursue the similar strategy and prove the Green-Griffiths conjecture for any complex surface $X$ of general type, one needs to know the existence of a 1-dimensional foliation directing any given Zariski dense entire curve $f:\cb\rightarrow X$. The condition of $c_1(X)^2-c_2(X)>0$ ensures the existence of multi-foliation 
on $X$ such that any entire curve should be tangent to it. The difficulty in 
proving the general case is that, we can not ensure that there exists such a 
(multi)-foliation on $X$ itself. However, inspired by a very recent work of Demailly 
\cite{Dem15}, we believe that his definition of a variety ``strongly of general type" is in some sense akin to the construction of foliations. Although one cannot construct foliations on $X$ directly, one can prove the existence of some special multi-foliations in certain Demailly-Semple tower of $X$. Indeed, in 
\cite{Dem10} the following theorem has been proved:
\begin{thm}
	Let $(X,V)$ be a directed variety of ``general type" (cf. \cite{Dem12} for the 
	definition of general type when $V$ is singular), then $\oc_{X_k}(m)\otimes 
	\pi^*_{k,0}\oc(-\frac{m}{kr}(1+\frac{1}{2}+\ldots+\frac{1}{k})A)$ is big thus 
	has sections for $m\gg k\gg 1$, where $X_k$ is the k-th stage of 
	Demailly-Semple tower of $X$ and $A$ is an ample divisor on $X$.
\end{thm}
By the Fundamental Vanishing theorem we know that for every entire curve 
$f:\cb\rightarrow X$, the $k$-jet $f_k:\cb\rightarrow X_k$ satisfies
$$
f_k(\cb)\subset \text{Bs} (H^0(X_k,\oc_{X_k}(m)\otimes \pi^*_{k,0} 
A^{-1}))\subsetneq X_k. 
$$

Assume that we have an entire curve $f:(\cb,T_{\cb})\rightarrow (X,T_X))$ such that its image in $X$ is Zariski dense. By the above 
theorem of Demailly, there exists an $N>0$ such that the lift of $f$ on the 
$N$th-stage Demailly-Semple tower can not be Zariski dense in $X_N$, therefore 
we can find an integer $k\geq0$ such that $f_j$ is Zariski dense in $X_j$ for 
each $0\leq j\leq k$, while the Zariski closure of the image of $f_{k+1}$ is 
$Z\subsetneq X_{k+1}$ which project onto $X_k$. Since $\text{rank}T_X=2$, $Z$ is 
a divisor of $X_{k+1}$. From the relation between $\text{Pic}(X_k)$ and 
$\text{Pic}(X)$ we know that $\oc_{X_k}(Z)\simeq \oc_{X_k}(\mathbf{a})\otimes 
\pi^*_{k,0}(B)$, for some $B\in \text{Pic}(X)$, $\mathbf{a}\in \mathbb{Z}^k$ 
and $a_k=m$. Therefore the projection $\pi_{k+1,k}:Z\rightarrow X_k$ is a 
ramified $m:1$ cover, which defines a rank $1$ multi-foliation $\fc_k\subset 
V_k$ on $X_k$, and $f_k:(\cb,T_\cb)\rightarrow (X_k,V_k)$ is tangent to this 
foliation. We define the linear subspace $W\subset T_Z\subset T_{X_{k+1}}|_Z$ to 
be the closure
$$
W:=\overline{T_{Z'}\cap V_{k+1}}
$$
taken on a suitable Zariski open set $Z'\subset Z_{\text{reg}}$ where the 
intersection $T_{Z'}\cap V_{k+1}$ has constant rank and is a subbundle of 
$T_{Z'}$. As is observed in \cite{Dem15}, we know that $\text{rank}W=1$ which 
is an 1-dimensional foliation. We first resolve the singularities of $Z$ to get 
a birational model $(\widetilde{Z},\widetilde{\fc})$ of $(Z,W)$ such that $\widetilde{Z}$ is smooth, then by the assumption in Theorem \ref{Green} we take a further finite sequence of blow-ups to get a new birational model $(Y,\fc)$ of $(\widetilde{Z},\widetilde{\fc})$, such that $\fc$ has only weakly reduced singularities. We now obtain a  
generically finite morphism $p:Y\rightarrow X_k$,  and the lift of $f$ to $Y$ denoted by $g:\cb\rightarrow Y$ is still a Zariski dense curve 
 tangent to $\fc$ satisfying $g=p\circ f_k$. Then we have
$$
K_Y\sim p^*K_{X_k}+R,
$$
where $R$ is an effective divisor whose support is contained in the 
ramification locus of $p$. We will call $X_k$ the \emph{critical Demailly-Semple 
	tower for $f$}. 

Now we state our conjectures about reduction of singularities to weakly reduced ones, and the generalization of Brunella Theorem to higher dimensional manifolds:
\begin{conjecture}\label{reduction of singularity}
Let $X$ be a K\"ahler manifold equipped with a foliation $\fc$ by curves. Then one can obtain a new birational model $(\widetilde{X},\widetilde{\fc})$ of $(X,\fc)$ by taking finite blowing-ups such that $\widetilde{\fc}$ has weakly reduced singularities.
\end{conjecture}

\begin{rem}
	From Proposition \ref{summary ais} it is easy to show that foliations with absolutely isolated singularities can be resolved into weakly reduced ones after finite blowing-ups.
\end{rem}

\begin{conjecture}\label{generalize Brunella}
	Let $(X,\fc)$ be a K\"ahler 1-foliated pair. Suppose that there is a Zariski dense entire curve $f:\cb\rightarrow X$ tangent to $\fc$, then we have
$$
T[f]\cdot c_1(\det N_\fc)\geq 0.
$$
\end{conjecture}

\begin{rem}
	If the singular set of $\fc$ is not discrete, it is difficult to construct a smooth 2-form in $c_1(\det N_\fc)$ as that appearing in Baum-Bott Formula (see Chapter 3 in \cite{Bru04}). Probably we should find some representation in the leafwise cohomology, i.e. cohomology group for laminations.
\end{rem}

Based on the conjectures above, we can prove the Green-Griffiths conjecture for complex surfaces:
\begin{proof}[Proof of Theorem \ref{Green}]
	Since we have
	\begin{align}\nonumber
	\det T_{X_k}=k\pi^*_{k,0}&\det T_X\otimes \oc_{X_k}(k+1,k,\ldots,2),\\\nonumber
	{\pi_{k,j}}_*T[f_{k}]&=T[f_j] \ \ \text{for}\ \ k\geq j,
	\end{align}
	by the tautological inequality and the condition of general type we have
	$$
	\langle T[f_k],\det T_{X_k}\rangle=-\sum_{j=1}^{k}(k-j+2)\langle  T[f_j], 
	\oc_{X_j}(-1)\rangle -k\langle T[f], K_X\rangle<0.
	$$
    Thus we obtain
	$$
	\langle T[g], K_Y\rangle=\langle T[f_k], K_{X_k}\rangle+\langle T[g], R\rangle >0.
	$$
	Conjecture \ref{reduction of singularity} tells us that we can find a new  birational pair $(\hat{Y},\hat{\fc})$ of $(Y,\fc)$ with weakly reduced singularities, then by Theorem \ref{negative normal main} we 
   have
 $$
 \langle T[\hat{g}], c_1(N_{\hat{\fc}})\rangle<0,
 $$
which is a contradiction to Conjecture \ref{generalize Brunella}, thus any entire curve must be algebraic degenerate.
\end{proof}

\noindent

\textbf{Acknowledgements}
I would like to thank my supervisor Professor Jean-Pierre Demailly for 
suggesting me to consider this problem, and for providing
numerous ideas during the course of stimulating exchanges. 
I would like to thank also Professor S\'ebastien Boucksom and Zhiyu Tian for further discussions about this work, and  Professor Sen Hu for his constant ecouragement. This research 
is supported by the China Scholarship Council.

\textsc{Institut Fourier \& University of Science and Technology of China} \\
\textsc{Ya Deng} \\
\verb"Email: Ya.Deng@fourier-grenoble.fr"

\end{document}